\documentclass[final,leqno]{siamltex}

\usepackage{amsfonts,amssymb} %%%add
\usepackage{amsmath} %%%add
\usepackage{color} %%%add
\usepackage{diagbox}
\usepackage{graphicx} %%%add
\usepackage{epsfig,mathrsfs} %%%add
\usepackage[bf,SL,BF]{subfigure} %%%add
\usepackage{fancyhdr} %%%add
\usepackage{CJK} %%%add
\usepackage{caption} %%%add
\usepackage{wrapfig} %%%add
\usepackage{cases} %%%add
\usepackage{bm}%%%add
\usepackage{algorithm}
\usepackage{algorithmic}
\usepackage{booktabs}
\usepackage{mathtools}

\title{Correction of BDFk for  fractional Feynman-Kac equation with L\'{e}vy flight
\thanks{This work was supported by NSFC 11601206. }}
\author{Jiankang Shi \thanks{
School of Mathematics and Statistics, Gansu Key Laboratory of Applied Mathematics and Complex Systems,
 Lanzhou University, Lanzhou 730000, P.R. China  (Email: shijk17@lzu.edu.cn)}
\and Minghua Chen\thanks{School of Mathematics and Statistics, Gansu Key Laboratory of Applied Mathematics and Complex Systems,
 Lanzhou University, Lanzhou 730000, P.R. China  (Email: chenmh@lzu.edu.cn)}
}

\begin{document}

\maketitle

\begin{abstract}
In this work, we present the correction formulas of the $k$-step BDF convolution quadrature at the starting $k-1$ steps for the  fractional Feynman-Kac equation with L\'{e}vy flight.
 The desired $k$th-order convergence rate can be achieved with nonsmooth data.
Based on the idea of [{\sc Jin, Li, and Zhou},
SIAM J. Sci. Comput., 39 (2017), A3129--A3152], we  provide a detailed convergence analysis for the correction BDF$k$ scheme. The numerical experiments with spectral method are given to illustrate the effectiveness of the presented method. To the best of our knowledge, this is the first proof of the convergence analysis and numerical verified the sapce fractional evolution equation with correction BDF$k$.
\end{abstract}

\begin{keywords}
Fractional Feynman-Kac equation with L\'{e}vy flight,  correction  of BDF$k$, fractional substantial derivative, error estimates.
\end{keywords}

%\begin{AMS}
%45B05, 65L60, 65M12
%\end{AMS}

\pagestyle{myheadings}
\thispagestyle{plain}
\markboth{J. SHI AND M.  CHEN}{CORRECTION OF BDFK FOR FRACTIONAL FEYNMAN-KAC   EQUATION}
\section{Introduction}
Functionals of Brownian motion have diverse applications in physics, mathematics, and other fields.
The probability density function of Brownian functionals satisfies the Feynman-Kac formula, which is  a Schr\"{o}dinger equation in imaginary time. The functionals of non-Brownian motion, or anomalous diffusion, follow the general fractional Feynman-Kac equation  \cite{CB:11,Carmi:10}, where the fractional substantial derivative is involved \cite{FJBE:06}. This paper focuses on providing the correction of  $k$-step backward differential formulas (BDF$k$) for backward  fractional Feynman-Kac equation   with L\'{e}vy flight \cite{CB:11,Carmi:10,ChenD:15,CD:18}
\begin{equation}\label{1.1}
\begin{split}
&^{C}_{s}D^{\gamma}_{t}G(t)+A G(t)=f(t)~~{\rm with}~~A:=(-\Delta)^{\alpha/2},
\end{split}
\end{equation}
where  $f$ is a given function and the initial condition $G(0)=G_0$ with the homogeneous Dirichlet boundary conditions.
Here $(-\Delta)^{\alpha/2}$ with $\alpha \in (1,2)$ is the fractional Laplacian, the definition of which is based on the spectral decomposition of the Dirichlet Laplacian \cite{LPGS:20,Thomee:06};
and the Caputo fractional substantial derivative with $0<\gamma<1$ is defined by \cite{CD:15,ChenD:15,FJBE:06}
\begin{equation}\label{1.20001}
\begin{split}
  ^{C}_{s}D^{\gamma}_{t}G(t)
    = \frac{1}{\Gamma(1-\gamma)}\int^{t}_{0}{\frac{e^{-\sigma(t-s)}}{(t-s)^{\gamma}}\left(\sigma + \frac{\partial}{\partial s}\right)G(s)}ds
%  \quad{\rm and}\quad  ^{C}D^{\gamma,\sigma}_{t}e^{-\sigma t}u(t)= e^{-\sigma t} \left({^{C}\partial^{\gamma}_{t}}u(t)\right)
\end{split}
\end{equation}
with  a constant $\sigma>0$. It should be noted that $\eqref{1.20001}$ reduces  to the Caputo fractional  derivative if $\sigma=0$.

High-order schemes   for the time discretization of \eqref{1.1} with $\sigma=0$ (Caputo fractional  derivative) have been proposed by various authors.
There are two predominant  discretization techniques in time direction:  L1-type approximation \cite{Lin:07,Lv:16,OS:74,Sun:06} and Lubich-Gr\"{u}nwald-Letnikov
approximation \cite{CLTA:07,Lubich:86,Meerschaert:04}.
For the first group, under the time regularity assumption on the solution, they developed the L1 schemes \cite{Lin:07,Sun:06} for the  Caputo fractional derivative   and strictly proved the stability and convergence rate with $\mathcal{O}\left(\tau^{2-\alpha} \right)$. It is   extended to the quadratic interpolation case  \cite{Lv:16} with a convergence rate  $\mathcal{O}\left(\tau^{3-\alpha} \right)$.
Recently, for a layer or blows up at $t=0$, a sharp new discrete stability result has been considered in \cite{SOG:17}.
In the second group, using fractional linear multistep method and Fourier transform, error analysis of up to sixth order temporal accuracy for fractional ordinary differential equation has been discussed \cite{Lubich:86} with the starting quadrature weights schemes. A few years later, based on operational calculus with sectorial operator,
 nonsmooth data error estimates for fractional evolution equations have been studied in \cite{Cuesta:06,Lu:96} and developed  in \cite{JLZ:16,JLZ:17} to restore   $\mathcal{O}\left(\tau^k \right)$.
Under the time regularity assumption, high order finite difference method  (BDF2) for the anomalous-diffusion equation has been studied in \cite{LiD:13} by analyzing the properties of the coefficients.
Application of  Grenander-Szeg\"{o} theorem, stability and convergence for time-fractional sub-diffusion equation have been provided in \cite{Hao:15,Ji:15} with  weighted and shifted Gr\"{u}nwald operator.

In recent years, the numerical method for backward  fractional Feynman-Kac equation  \eqref{1.1} with $\alpha=2$ were developed. For example,
the time discretization of Caputo fractional substantial derivative  was first provided in \cite{CD:15} with the starting quadrature weights schemes.
Spectral methods for substantial fractional ordinary differential equations was presented in \cite{HZS:18}.
Under smooth assumption, numerical algorithms (finite difference and finite element) for   \eqref{1.1} with $\alpha=2$ are considered in \cite{DCB:15}.
Moreover, the second-convergence  analysis are discussed in \cite{CD:18,HCL:17}. In addition, the problem with nonsmooth solution is also discussed in  \cite{CD:18}.
Recently, the second-order error estimates are presented in \cite{SND:2020} with nonsmooth initial data.
However, it seems that there are no published works for more than three order  accurate scheme  for model \eqref{1.1} with L\'{e}vy flight.
In this work, we  provide a detailed convergence analysis of the correction BDF$k$  ($k\leq 6$) for \eqref{1.1} with nonsmooth data.

We first provide the solutions of  fractional Feynman-Kac equation with L\'{e}vy flight.
\subsection*{Solution representation for \eqref{1.1}}
Let $\Omega$ be a bounded domain with a boundary $\partial\Omega$.
If $\sigma=0$, then  (\ref{1.1}) reduce to the following time-space Caputo-Riesz fractional diffusion equation \cite{CDW:13}
\begin{equation*}
\left\{
\begin{split}
&^{C}\partial^{\gamma}_{t}G(x,t)+A G(x,t)=f(x,t), &\quad (x,t)&\in \Omega \times (0,T]  \\
&G(x,0)=v(x),  &\quad x&\in \Omega \\
&G(x,t)=0, &\quad (x,t)&\in \partial\Omega \times (0,T],
\end{split}\right.
\end{equation*}
where $f$ is a given function and $A$ denotes the Laplacian $(-\Delta)^{\alpha/2}$.
Based on the idea of \cite{JLZ:16,SY:11,Thomee:06} with the  eigenpairs $\{ (\lambda^{\alpha/2}_{j},\varphi_{j}) \}^{\infty}_{j=1}$ of the operator $A$, it  is easy to get
\begin{equation}\label{1.3}
G(x,t)=E(t)v+\int^{t}_{0}{\overline{E}(t-s)f(s)}ds.
\end{equation}
Here the operators
$E(t)$ and $\overline{E}(t)$ are, respectively, given by
$$E(t)=\sum^{\infty}_{j=1}E_{\gamma,1}(-\lambda^{\alpha/2}_{j}t^{\gamma})(v,\varphi_{j})\varphi_{j}(x)$$
and
$$\overline{E}(t){\chi}=\sum^{\infty}_{j=1}t^{\gamma-1}E_{\gamma,\gamma}(-\lambda^{\alpha/2}_{j}t^{\gamma})(v,\varphi_{j})(\chi,\varphi_{j})\varphi_{j}(x)$$
with the Mittag-Leffler function $E_{\gamma,\alpha/2}(z)$  \cite[p.\,17]{Podlubny:99}, i.e.,
$$E_{\gamma,\alpha/2}(z)=\sum^{\infty}_{k=0}\frac{z^{k}}{\Gamma(k\gamma+\alpha/2)},z\in\mathbb{C}.$$

If $\sigma>0$, using  \eqref{1.3} and the following  property \cite{CD:15,ChenD:15}
\begin{equation*}
\begin{split}
^{C}_{s}D^{\gamma}_{t}e^{-\sigma t}G(t)= e^{-\sigma t} \left({^{C}\partial^{\gamma}_{t}}G(t)\right),
\end{split}
\end{equation*}
infer that
\begin{equation}\label{1.4}
G(x,t)=e^{-\sigma t}E(t)v+e^{-\sigma t}\int^{t}_{0}{\overline{E}(t-s)e^{\sigma s}f(s)}ds.
\end{equation}

From the above solution representation of \eqref{1.1}, we known that   the smoothness of all the data of \eqref{1.1} do not imply the smoothness of the solution $G$.
For example, if $G_0 \in L^2(\Omega)$ and $\gamma \in (0,1)$ with $\alpha=2$, the following estimate holds  \cite[Theorem 2.1]{SY:11}
$$\|^C\partial_t^\gamma G(t)\|_{L^2(\Omega)}\leq ct^{-\gamma}\| G_0\|_{L^2(\Omega)},$$
which reduces to a classical case $\|\partial_t G(t)\|_{L^2(\Omega)}\leq ct^{-1}\| G_0\|_{L^2(\Omega)}$ if $\gamma=1$ \cite[Lemma 3.2 ]{Thomee:06}.
This shows that $G$ has an initial layer  at $t\rightarrow 0^{+}$ (i.e., unbounded near $t=0$) \cite{SOG:17}.
Hence, the high-order convergence rates may not hold for nonsmooth data.
Thus, the corrected algorithms are necessary in order to restore the desired convergence rate, even for smooth initial data.
In this paper, based on the idea of \cite{JLZ:17}, we present the correction  of the $k$-step BDF convolution quadrature (CQ) at the starting $k-1$ steps for the backward fractional Feynman-Kac equation
with L\'{e}vy flight \eqref{1.1}.  The desired $k$th-order convergence rate can be achieved with nonsmooth data.
To the best of our knowledge, this is the first proof of the convergence analysis and numerical verified the sapce fractional evolution equation with correction BDF$k$.

The paper is organized as follows. In the next Section, we provide  the correction of the $k$-step BDF convolution quadrature at the starting $k-1$ steps for \eqref{1.1}.
In Section 3, based on operational calculus,  the detailed convergence analysis of the correction BDF$k$ are provided. To show the effectiveness of the presented schemes, the results of numerical experiments are reported in Section 4.

\section{Correction of BDF$k$}\label{sec:1}
Let $t_{n}=n \tau, n=0,1,\ldots,N$ with $\tau=\frac{T}{N}$ the uniform time steplength, and let $G^{n}$ denote the approximation of $G(t)$ and $f^{n}=f(t_{n})$. The convolution quadrature  generated by BDF$k$, $k=1,2,\ldots,6,$ approximates the Riemann-Liouville fractional substantial derivative $_sD^{\gamma}_{t}$ by \cite{CD:15}
\begin{equation}\label{2.1}
  \overline{D}^{\gamma}_{\tau}\varphi^{n}:=\frac{1}{\tau^{\gamma}}\sum^{n}_{j=0}q_{j}\varphi^{n-j}
\end{equation}
with $\varphi^{n}=\varphi(t_{n})$. Here  the weights $q_{j}=e^{-\sigma j \tau}b_{j}$ and $b_{j}$ are the coefficients in the series expansion
\begin{equation}\label{2.2}
  \delta^{\gamma}_{\tau}(\xi)=\frac{1}{\tau^{\gamma}}\sum^{\infty}_{j=0}b_{j}\xi^{j} \quad{\rm with} \quad \delta_{\tau}(\xi):=\frac{1}{\tau}\sum^{k}_{j=1}\frac{1}{j}(1-\xi)^{j}
  \quad{\rm and}~~\delta(\xi):=\delta_{1}(\xi).
\end{equation}
%and we write $\delta(\xi)=\delta_{1}(\xi)$, i.e., with $\tau=1$.
Then the standard BDF$k$ for \eqref{1.1} is as following
\begin{equation}\label{2.3}
  \overline{D}^{\gamma}_{\tau} \left(G^{n}-e^{-\sigma t_{n}} G(0)\right) + AG^{n}=f(t_{n}).
\end{equation}

To obtain the $k$-order accuracy with nosmooth data, we correct the standard   BDF$k$ \eqref{2.3} at the starting $k-1$ steps by
\begin{equation}\label{2.4}
\begin{split}
    \overline{D}^{\gamma}_{\tau} \left(G^{n}-e^{-\sigma t_{n}} G(0)\right) + AG^{n}=&-a^{(k)}_{n}e^{-\sigma n \tau}AG^{0}+b^{(k)}_{n}f(0)\\
         &+\sum^{k-2}_{l=1}d^{(k)}_{l,n}\tau^{l}\partial^{l}_{t}f(0)+f(t_{n}) \quad 1\le n \le k-1;\\
    \overline{D}^{\gamma}_{\tau} \left(G^{n}-e^{-\sigma t_{n}} G(0)\right) + AG^{n}=&f(t_{n}) \quad k \le n \le N.
\end{split}
\end{equation}
Here the correction  coefficients $a^{(k)}_{n}$  and $b^{(k)}_{n}$ are given in  Table \ref{table:1} and $d^{(k)}_{l,n}$ is given in Table \ref{table:2}.
We  noted that the correction coefficients share similarities with  the fractional Caputo equations \cite{JLZ:17}.
\subsection{Solution representation with CQ for \eqref{1.1}}
\begin{table}[h]\fontsize{8.5pt}{12pt}\selectfont%Éú³É¸¡¶¯±í¸ñ
\begin{center}%\def\tabcolsep{28.5pt}%±í¸ñ¾ÓÖÐ
\caption {The coefficients $a^{(k)}_{n}$ and $b^{(k)}_{n}$.} \vspace{5pt}%±êÌ⣬Àë±í¸ñÒ»¶¨µÄ¾àÀë
\begin{tabular}{|c| c c c c c|c c c c c|}                                \hline  %»­¶¥¶ËµÄºáÏß
  Order of BDF & $a^{(k)}_{1}$       & $a^{(k)}_{2}$       & $a^{(k)}_{3}$     & $a^{(k)}_{4}$       & $a^{(k)}_{5}$    &
                 $b^{(k)}_{1}$       & $b^{(k)}_{2}$       & $b^{(k)}_{3}$     & $b^{(k)}_{4}$       & $b^{(k)}_{5}$      \\ \hline
  $k=2$        & $\frac{1}{2}$       &                     &                   &                     &                  &
                 $\frac{1}{2}$       &                     &                   &                     &                    \\ \hline
  $k=3$        & $\frac{11}{12}$     & $-\frac{5}{12}$     &                   &                     &                  &
                 $\frac{11}{12}$     & $-\frac{5}{12}$     &                   &                     &                    \\ \hline
  $k=4$        & $\frac{31}{24}$     & $-\frac{7}{6}$      & $\frac{3}{8}$     &                     &                  &
                 $\frac{31}{24}$     & $-\frac{7}{6}$      & $\frac{3}{8}$     &                     &                    \\ \hline
  $k=5$        & $\frac{1181}{720}$  & $-\frac{177}{80}$   & $\frac{341}{240}$ & $-\frac{251}{720}$  &                  &
                 $\frac{1181}{720}$  & $-\frac{177}{80}$   & $\frac{341}{240}$ & $-\frac{251}{720}$  &                    \\ \hline
  $k=6$        & $\frac{2837}{1440}$ & $-\frac{2543}{720}$ & $\frac{17}{5}$    & $-\frac{1201}{720}$ & $\frac{95}{288}$ &
                 $\frac{2837}{1440}$ & $-\frac{2543}{720}$ & $\frac{17}{5}$    & $-\frac{1201}{720}$ & $\frac{95}{288}$   \\ \hline % »­µ×¶ËµÄºáÏß
%%%%%%%%%%%%%%%%%%%%%%%%%%%%%%%%%%%%%%%%%%%%%%%%%%%%%%%%%%%%%%%%%%%%%%%%%%%%%%%%%%%%%%%%%%%%%%%%%%%%%%%%%%%%%%%%%%%%%%%%%%%%%%%%%%%
\end{tabular}\label{table:1}%\vspace{-15pt}
\end{center}
\end{table}
\begin{table}[h]\fontsize{9pt}{12pt}\selectfont%Éú³É¸¡¶¯±í¸ñ
 \begin{center}%\def\tabcolsep{28.5pt}%±í¸ñ¾ÓÖÐ
  \caption {The coefficients $d^{(k)}_{l,n}$.} \vspace{5pt}%±êÌ⣬Àë±í¸ñÒ»¶¨µÄ¾àÀë
\begin{tabular}{|c c| c c c c c|}                                \hline  %»­¶¥¶ËµÄºáÏß
  Order of BDF &       & $d^{(k)}_{l,1}$     & $d^{(k)}_{l,2}$     & $d^{(k)}_{l,3}$   & $d^{(k)}_{l,4}$     & $d^{(k)}_{l,5}$  \\ \hline
  $k=3$        & $l=1$ & $\frac{1}{12}$      &        0            &                   &                     &                  \\ \hline
  $k=4$        & $l=1$ & $\frac{1}{6}$       & $-\frac{1}{12}$     &       0           &                     &                  \\
               & $l=2$ &      0              &        0            &       0           &                     &                  \\ \hline
  $k=5$        & $l=1$ & $\frac{59}{240}$    & $-\frac{29}{120}$   & $\frac{19}{240}$  &       0             &                  \\
               & $l=2$ & $\frac{1}{240}$     & $-\frac{1}{240}$    &       0           &       0             &                  \\
               & $l=3$ & $-\frac{1}{720}$     &       0             &       0           &       0             &                  \\ \hline
  $k=6$        & $l=1$ & $\frac{77}{240}$    & $-\frac{7}{15}$     & $\frac{73}{240}$  & $-\frac{3}{40}$     &      0           \\
               & $l=2$ & $\frac{1}{96}$      & $-\frac{1}{60}$     & $\frac{1}{160}$   &       0             &      0           \\
               & $l=3$ & $-\frac{1}{360}$     & $\frac{1}{720}$     &       0           &       0             &      0           \\
               & $l=4$ &       0             &       0             &       0           &       0             &      0           \\\hline % »­µ×¶ËµÄºáÏß
%%%%%%%%%%%%%%%%%%%%%%%%%%%%%%%%%%%%%%%%%%%%%%%%%%%%%%%%%%%%%%%%%%%%%%%%%%%%%%%%%%%%%%%%%%%%%%%%%%%%%%%%%%%%%%%%%%%%%%%%%%%%%%%%%%%
    \end{tabular}\label{table:2}%\vspace{-15pt}
  \end{center}
\end{table}

First we split right-hand side $f$ into
\begin{equation}\label{2.5}
f(t)=f(0)+\sum^{k-2}_{l=1}\frac{t^{l}}{l!}\partial^{l}_{t}f(0)+R_{k}.
\end{equation}
Here
\begin{equation}\label{2.6}
R_{k}=f(t)-f(0)-\sum^{k-2}_{l=1}\frac{t^{l}}{l!}\partial^{l}_{t}f(0)=\frac{t^{k-1}}{(k-1)!}\partial^{k-1}_{t}f(0)
       +\frac{t^{k-1}}{(k-1)!} \ast \partial^{k}_{t}f,
\end{equation}
and the symbol $\ast$ denotes Laplace convolution. Let $W(t):=G(t)-e^{-\sigma t} G(0)$ with $W(0)=0$.
Then we can rewrite \eqref{1.1} as
\begin{equation}\label{2.7}
  _{s}D^{\gamma}_{t}W(t) + AW(t)=-A e^{-\sigma t} G(0)+ f(0)+\sum^{k-2}_{l=1}\frac{t^{l}}{l!}\partial^{l}_{t}f(0)+R_{k}
\end{equation}
with $ ^{C}_{s}D^{\gamma}_{t}G(t)=_s\!D^{\gamma}_{t}\left[G(t)-e^{-\sigma t} G(0)\right]=_s\!D^{\gamma}_{t}W(t)$ in \cite{CD:15}.

Applying Laplace transform in \cite{CD:15} to both sides of the above equation, we have
\begin{equation*}
  (\sigma+z)^{\gamma}\widehat{W}(z)+A\widehat{W}(z)=-A(\sigma+z)^{-1} G(0)+z^{-1} f(0)+\sum^{k-2}_{l=1}\frac{1}{z^{l+1}}\partial^{l}_{t}f(0)+\widehat{R}_{k}(z),
\end{equation*}
where $\widehat{W}$ and $\widehat{R}_{k}$ denote the Laplace transform, i.e, $\widehat{u}(z)=\int^{\infty}_{0}{e^{-zt}u(t)}dt$. %\eqref{1.00003}.
Then
\begin{equation*}
  \widehat{W}(z) = \left( \left(\sigma+z\right)^{\gamma}+A  \right)^{-1}\left[- A (\sigma+z)^{-1} G(0)+z^{-1} f(0) +\sum^{k-2}_{l=1}\frac{1}{z^{l+1}}\partial^{l}_{t}f(0)+\widehat{R}_{k}(z) \right].
\end{equation*}
By the inverse Laplace transform, we can obtain the solution $W(t)$ as following
\begin{equation}\label{2.10}
\begin{split}
W(t)= & \frac{1}{2\pi i} \int_{\Gamma_{\theta,\kappa}} {e^{zt}K(\sigma+z)\left( -A G(0)+\frac{\sigma+z}{z} f(0)  \right) } dz \\
      & +\frac{1}{2\pi i}\int_{\Gamma_{\theta,\kappa}}{e^{zt}(\sigma+z)K(\sigma+z)
      \left(\sum^{k-2}_{l=1}\frac{1}{z^{l+1}}\partial^{l}_{t}f(0)+\widehat{R}_{k}(z)\right)}dz
\end{split}
\end{equation}
with
\begin{equation}\label{2.11}
  K(\sigma+z)=\left( \left(\sigma+z\right)^{\gamma}+A  \right)^{-1}(\sigma+z)^{-1}.
\end{equation}
Here the $\Gamma_{\theta,\kappa}$ is defined by \cite{DLQW:18}
\begin{equation}\label{2.a10}
\Gamma_{\theta,\kappa}=\{ z \in \mathbb{C}:|z|=\kappa, |\arg z| \le \theta \} \cup \{ z \in \mathbb{C}: z=r e^{\pm i\theta},\kappa\le r <\infty \}.
\end{equation}
It should be noted that we have  the following resolvent bound \cite{Cuesta:06}
\begin{equation}\label{2.11111}
||(z^{\gamma}+A)^{-1}||\le c|z|^{-\gamma} \quad {\rm and} \quad ||K(z)||\le c|z|^{-1-\gamma}
\end{equation}
with a positive constant $c$.

%\begin{figure}
%  \centering
%  \includegraphics[scale=0.5]{222.eps}\\
%  \caption{234}\label{2.1111111}
%\end{figure}

%\begin{figure}
%  \centering
%  % Requires \usepackage{graphicx}
%  \includegraphics[width=0.50\textwidth]{222.eps}\\
%  \caption{1234}\label{1234}
%\end{figure}

%%%%%%%%%%%%%%%%%%%%%%%%%%%%%%%%%%%%%%%%%%%%%%%%%%%%%%%%%%%%%%%%%%%%%%%%%%%%%%%%%%%%%%%%%%%%%%%%%%%%%%

\subsection{Discrete solution representation with CQ for \eqref{2.4}}
In this subsection, we provide the discrete solution  of \eqref{2.4}.
\begin{lemma}\label{Lemma 2.1}
Let $f\in C^{k-1}([0,T];L^{2}(\Omega))$ and $\int^{t}_{0}{(t-s)^{\gamma-1}}||\partial^{l}_{s}f(s)||_{L^{2}(\Omega)}ds<\infty$. Let discrete solution $W^{n}=G^{n}-e^{-\sigma n \tau}G(0)$ with $W^{0}=0$. Then
\begin{equation}\label{2.13}
\begin{split}
W^{n}
%=& \frac{1}{2 \pi i} \int_{\Gamma^{\tau}_{\theta,\kappa}}{e^{t_{n}z} K\left(\delta_{\tau}(e^{-(\sigma+z)\tau})\right)\delta(e^{-(\sigma+z)\tau})
%  \left( \frac{e^{-(\sigma+z)\tau}}{1-e^{-(\sigma+z)\tau}}+\sum^{k-1}_{j=1} a^{(k)}_{j}e^{-(\sigma+z) j \tau} \right) AG(0)} dz\\
%&+\frac{1}{2 \pi i} \int_{\Gamma^{\tau}_{\theta,\kappa}}{e^{t_{n}z} K\left(\delta_{\tau}(e^{-(\sigma+z)\tau})\right)\delta(e^{-(\sigma+z)\tau})
%  \left( \frac{e^{-z\tau}}{1-e^{-z\tau}}+\sum^{k-1}_{j=1} b^{(k)}_{j}e^{-z j \tau} \right)f(0)} dz \\
%&+\frac{1}{2 \pi i} \int_{\Gamma^{\tau}_{\theta,\kappa}}{e^{t_{n}z} K\left(\delta_{\tau}(e^{-(\sigma+z)\tau})\right)\delta_{\tau}(e^{-(\sigma+z)\tau})
%  \sum^{k-2}_{l=1}\left(\sum^{n=1}_{\infty}e^{-zn\tau}\frac{n^{l}}{l!}+\sum^{k-1}_{j=1} d^{(k)}_{l,j}e^{-zj\tau}\right)\tau^{l+1}\partial^{l}_{t}f(0)}dz\\
%&+\frac{1}{2 \pi i} \int_{\Gamma^{\tau}_{\theta,\kappa}}{e^{t_{n}z} K\left(\delta_{\tau}(e^{-(\sigma+z)\tau})\right)\delta_{\tau}(e^{-(\sigma+z)\tau})
%  \widetilde{R}_{k}(e^{-z\tau}) \tau } dz\\
=&\frac{1}{2 \pi i} \int_{\Gamma^{\tau}_{\theta,\kappa}} e^{t_{n}z} K\left(\delta_{\tau}(e^{-(\sigma+z)\tau})\right)\\
&\qquad\left[-\mu_{1}(e^{-(\sigma+z)\tau}) AG^{0}+\frac{\delta_{\tau}(e^{-(\sigma+z)\tau})}{\delta_{\tau}(e^{-z\tau})}\mu_{2}(e^{-z\tau})f(0)\right]dz \\
& +\frac{1}{2 \pi i} \int_{\Gamma^{\tau}_{\theta,\kappa}}e^{t_{n}z}K\left(\delta_{\tau}(e^{-(\sigma+z)\tau})\right)\delta_{\tau}(e^{-(\sigma+z)\tau})\\
&\qquad\times \sum^{k-2}_{l=1}\left(\frac{\gamma_{l}(e^{-z\tau})}{l!} +\sum^{k-1}_{j=1}d^{(k)}_{l,n}e^{-z j\tau}\right)\tau^{l+1}\partial^{l}_{t}f(0)dz \\
& +\frac{1}{2 \pi i} \int_{\Gamma^{\tau}_{\theta,\kappa}}e^{t_{n}z}K\left(\delta_{\tau}(e^{-(\sigma+z)\tau})\right)
  \delta_{\tau}(e^{-(\sigma+z)\tau})\tau\widetilde{R}_{k}(e^{-z\tau}) dz.
\end{split}
\end{equation}
Here  the contour is
\begin{equation*}
\Gamma^{\tau}_{\theta,\kappa}=\{z\in\mathbb{C}:|z|=\kappa,|\arg z|\le\theta\}\cup
\left\{z\in\mathbb{C}:z=re^{\pm i\theta},\kappa\le r<\frac{\pi}{\tau\sin(\theta)}\right\},
\end{equation*}
and $\widetilde{R}_{k}(e^{-z\tau})= \sum^{\infty}_{n=1} e^{-zn\tau} R_{k}(t_{n})$,
and
\begin{equation}\label{2.14}
\begin{split}
\mu_{1}(\xi)&=\delta(\xi) \left( \frac{ \xi}{1- \xi} + \sum^{k-1}_{j=1}a^{(k)}_{j}\xi^{j}\right),\\
\mu_{2}(\xi)&=\delta(\xi) \left( \frac{ \xi}{1- \xi} + \sum^{k-1}_{j=1}b^{(k)}_{j}\xi^{j}\right) ~~{\rm and}~~
\gamma_{l}(\xi)=\sum^{\infty}_{n=1}n^{l} \xi^{n}.
\end{split}
\end{equation}
\end{lemma}
\begin{proof}
Let $W^{n}=G^{n}-e^{-\sigma n \tau}G(0)$.  From \eqref{2.4}, it holds
\begin{equation}\label{5.1}
\begin{split}
    \overline{D}^{\gamma}_{\tau} W^{n}+ AW^{n}=&-(1+a^{(k)}_{n})e^{-\sigma n \tau}AG^{0}+(1+a^{(k)}_{n})f(0)\\
         &+\sum^{k-2}_{l=1}\left( \frac{t^{l}_{n}}{l!} + d^{(k)}_{l,n}\tau^{l}\right)\partial^{l}_{t}f(0)+R_{k}(t_{n}) \quad 1\le n \le k-1;\\
    \overline{D}^{\gamma}_{\tau} W^{n} + AW^{n}=&-e^{-\sigma n \tau}AG^{0} + f(0)
       + \sum^{k-2}_{l=1} \frac{t^{l}_{n}}{l!} \partial^{l}_{t}f(0)+R_{k}(t_{n})\quad k \le n \le N.
\end{split}
\end{equation}
Multiplying \eqref{5.1} by $\xi^{n}$ and summing over $n$, we have
\begin{equation}\label{5.2}
\begin{split}
& \sum^{\infty}_{n=1} \xi^{n}\left( \overline{D}^{\gamma}_{\tau} W^{n} + AW^{n}\right)
= -\left(\sum^{\infty}_{n=1} \xi^{n} e^{-\sigma n \tau} \!+\! \sum^{k-1}_{j=1} \xi^{j}a^{(k)}_{j}e^{-\sigma j \tau}\right) AG^{0}\\
&   +\left(\sum^{\infty}_{n=1} \xi^{n} \!+\! \sum^{k-1}_{j=1} \xi^{j}b^{(k)}_{j}\right)f(0)
  +\sum^{k-2}_{l=1}\left( \sum^{\infty}_{n=1} \xi^{n}\frac{t^{l}_{n}}{l!} + \sum^{k-1}_{j=1} \xi^{j}d^{(k)}_{l,n}\tau^{l}\right)\partial^{l}_{t}f(0)
   + \widetilde{R}_{k}(\xi) \\
&= -\left( \frac{\xi e^{-\sigma \tau}}{1-\xi e^{-\sigma \tau}} + \sum^{k-1}_{j=1} \xi^{j}a^{(k)}_{j}e^{-\sigma j \tau}\right) AG^{0}
   +\left(\frac{\xi }{1-\xi} + \sum^{k-1}_{j=1} \xi^{j}b^{(k)}_{j}\right)f(0) \\
 & +\sum^{k-2}_{l=1}\left(\frac{\gamma_{l}(\xi)}{l!} +\sum^{k-1}_{j=1} \xi^{j}d^{(k)}_{l,n}\right)\tau^{l}\partial^{l}_{t}f(0) + \widetilde{R}_{k}(\xi),
\end{split}
\end{equation}
where  $\widetilde{R}_{k}(\xi)= \sum^{\infty}_{n=1} \xi^{n} R_{k}(t_{n}) $ and $\gamma_{l}(\xi)=\sum^{\infty}_{n=1}n^{l} \xi^{n}$. Here we use the property  $\sum^{\infty}_{n=1} \xi^{n} = \frac{\xi}{1-\xi}$.
Since
\begin{equation}\label{5.3}
\begin{split}
 &\sum^{\infty}_{n=1} \xi^{n} \overline{D}^{\gamma}_{\tau} W^{n}
 = \sum^{\infty}_{n=1} \xi^{n} \frac{1}{\tau^{\gamma}} \sum^{n}_{j=1} q_{n-j} W^{j}
 = \sum^{\infty}_{j=1} \sum^{\infty}_{n=j}\xi^{n} \frac{1}{\tau^{\gamma}} q_{n-j} W^{j} \\
 =& \sum^{\infty}_{j=1} \sum^{\infty}_{n=0}\xi^{n+j} \frac{1}{\tau^{\gamma}} q_{n} W^{j}
 = \frac{1}{\tau^{\gamma}}\sum^{\infty}_{n=0}\xi^{n} q_{n}\sum^{\infty}_{j=1} \xi^{j} W^{j}
 =\delta^{\gamma}_{\tau}(e^{-\sigma\tau}\xi) \widetilde{W}(\xi)
\end{split}
\end{equation}
with $W^{0}=0$ and $\sum^{\infty}_{n=1} \xi^{n}AW^{n}=A\widetilde{W}(\xi)$. We obtain
\begin{equation}\label{5.4}
\begin{split}
\widetilde{W}(\xi)
& = K\left(\delta_{\tau}(e^{-\sigma\tau}\xi)\right)\delta_{\tau}(e^{-\sigma\tau}\xi)
\left[ - \left( \frac{ e^{-\sigma \tau}\xi}{1- e^{-\sigma \tau}\xi} \!+\! \sum^{k-1}_{j=1}a^{(k)}_{j}e^{-\sigma j \tau}\xi^{j}\right) AG^{0}
    \right. \\
& \quad\left.+\left(\frac{\xi }{1-\xi} \!+\! \sum^{k-1}_{j=1} b^{(k)}_{j}\xi^{j}\right)f(0)+ \sum^{k-2}_{l=1}\left(\frac{\gamma_{l}(\xi)}{l!} +\sum^{k-1}_{j=1}d^{(k)}_{l,n}\xi^{j}\right)\tau^{l}\partial^{l}_{t}f(0) + \widetilde{R}_{k}(\xi) \right]\\
&= K\left(\delta_{\tau}(e^{-\sigma\tau}\xi)\right)
\left[ -\tau^{-1}\mu_{1}(e^{-\sigma\tau}\xi) AG^{0}
   +\tau^{-1}\frac{\delta_{\tau}(e^{-\sigma\tau}\xi)}{\delta_{\tau}(\xi)}\mu_{2}(\xi)f(0)
    \right. \\
& \quad+\left. \delta_{\tau}(e^{-\sigma\tau}\xi)\sum^{k-2}_{l=1}\left(\frac{\gamma_{l}(\xi)}{l!} +\sum^{k-1}_{j=1}d^{(k)}_{l,n}\xi^{j}\right)\tau^{l}\partial^{l}_{t}f(0) + \delta_{\tau}(e^{-\sigma\tau}\xi)\widetilde{R}_{k}(\xi) \right],
\end{split}
\end{equation}
where $K$ is given by \eqref{2.11} and
\begin{equation}\label{5.5}
\begin{split}
\mu_{1}(e^{-\sigma\tau}\xi)&=\delta(e^{-\sigma\tau}\xi) \left( \frac{ e^{-\sigma \tau}\xi}{1- e^{-\sigma \tau}\xi} + \sum^{k-1}_{j=1}a^{(k)}_{j}e^{-\sigma j \tau}\xi^{j}\right),\\
\mu_{2}(\xi)&=\delta(\xi) \left( \frac{ \xi}{1- \xi} + \sum^{k-1}_{j=1}b^{(k)}_{j}\xi^{j}\right).
\end{split}
\end{equation}
According to  Cauchy's integral formula, and the change of variables $\xi=e^{-z\tau}$, and  Cauchy's theorem of complex analysis, we have
\begin{equation}\label{5.6}
\begin{split}
W^{n}
&=\frac{1}{2\pi i}\int_{|\xi|=\varrho_{\kappa}}{\xi^{-n-1} \widetilde{W}(\xi)} d\xi
=\frac{1}{2\pi i}\int_{\Gamma^{\tau}}{e^{t_{n}z}\tau \widetilde{W}(e^{-z\tau})} dz\\
&=\frac{1}{2\pi i}\int_{\Gamma^{\tau}_{\theta,\kappa}}{e^{t_{n}z} \tau\widetilde{W}(e^{-z\tau})} dz
\end{split}
\end{equation}
with the Bromwich contours \cite{Podlubny:99}
\begin{equation*}\label{2.1201}
\Gamma^{\tau}= \{z=\kappa+1+iy: y\in \mathbb{R} \quad{\rm and}\quad |y|\le \pi/\tau \}.
\end{equation*}
The proof is completed.
\end{proof}

\section{Convergence analysis}
In this section,  based on the idea of \cite{JLZ:17}, we provided  the detailed convergence analysis of the correction BDF$k$ for \eqref{1.1} with L\'{e}vy flight.
\subsection{A few technical Lemmas}
We first introduce   a few technical lemmas.
\begin{lemma}\label{Lemma 2.2}
Let $\delta(\xi)$ is given by \eqref{2.2} for $1\le k \le 6$. Then
\begin{equation*}
 \delta(e^{-y})=y-\frac{1}{k+1}y^{k+1}+\mathcal{O}(y^{k+2}).% \quad{\rm i.e.}\quad  \left|\delta(e^{-y})\right|\le |y|+ c(|y|^{k+1}).
\end{equation*}
\end{lemma}
\begin{proof}
According to \eqref{2.2} and the Taylor series expansion $e^{-y}=\sum^{\infty}_{n=0} {\frac{(-y)^{n}}{n!}}=\sum^{k+1}_{n=0} {\frac{(-y)^{n}}{n!}}+\mathcal{O}(y^{k+2}) $, we have
\begin{equation*}
\begin{split}
\delta(e^{-y})
&= 1-e^{-y} = y-\frac{1}{2}y^2+\mathcal{O}(y^3), \quad k=1;\\
\delta(e^{-y})
&= \frac{3}{2}\left( 1-\frac{4}{3}e^{-y}+\frac{1}{3}e^{-(2y)} \right)= y-\frac{1}{3}y^3+\mathcal{O}(y^4), \quad k=2;\\
\delta(e^{-y})
&= \frac{11}{6}\left( 1-\frac{18}{11}e^{-y}+\frac{9}{11}e^{-2y} -\frac{2}{11}e^{-3y}\right)= y-\frac{1}{4}y^4+\mathcal{O}(y^5), \quad k=3;\\
\delta(e^{-y})
&= \frac{25}{12}\left( 1-\frac{48}{25}e^{-y}+\frac{36}{25}e^{-2y}-\frac{16}{25}e^{-3y}+\frac{3}{25}e^{-4y} \right)
= y-\frac{1}{5}y^5+\mathcal{O}(y^6), ~~ k=4;\\
\delta(e^{-y})
&= \frac{137}{60}\left( 1-\frac{300}{137}e^{-y}+\frac{300}{137}e^{-2y}-\frac{200}{137}e^{-3y}+\frac{75}{137}e^{-4y}-\frac{12}{137}e^{-5y} \right)\\
&= y-\frac{1}{6}y^6+\mathcal{O}(y^7), \quad k=5;\\
\delta(e^{-y})
&= \frac{147}{60}\left( 1-\frac{360}{147}e^{-y}+\frac{450}{147}e^{-2y}-\frac{400}{147}e^{-3y}+\frac{225}{147}e^{-4y}-\frac{72}{147}e^{-5y}+\frac{10}{147}e^{-6y} \right)\\
&= y-\frac{1}{7}y^7+\mathcal{O}(y^8), \quad k=6.
\end{split}
\end{equation*}
The proof is completed.
%$$\delta(e^{-y}):=\sum^{k}_{j=1}\frac{1}{j}(1-e^{-y})^{j},$$
\end{proof}

\begin{lemma}\label{Lemma 2.3}%%%[\cite{JLZ:17}]
Let $\delta_{\tau}(\xi)$ be given by \eqref{2.2} for $1\le k \le 6$. Then there  exist the positive constants $c_{1},c_{2}$ and $c$ such that
\begin{equation*}
\begin{split}
c_{1}|z|\le |\delta_{\tau}(e^{-z\tau})| \le c_{2}|z| \quad {\rm and} \quad |\delta^{\gamma}_{\tau}(e^{-z\tau})-z^{\gamma}|\le c \tau^{k}|z|^{k+\gamma}
\quad{\forall} z\in \Gamma^{\tau}_{\theta,\kappa}.
\end{split}
\end{equation*}
\end{lemma}
\begin{proof}
From  Lemma \ref{Lemma 2.2}, we have $c_{1}|z|\le |\delta_{\tau}(e^{-z\tau})| \le c_{2}|z|$.
On the other hand, according to  (39) of  \cite{JLZ:17} and Lemma \ref{Lemma 2.2}, we have
\begin{equation*}
|\delta^{\gamma}_{\tau}(e^{-z\tau})-z^{\gamma}| = \gamma \left| \int^{\delta_{\tau}(e^{-z\tau})}_{z} { \zeta^{\gamma-1} } d\zeta\right|
\leq \max_{\zeta} |\zeta|^{\gamma-1} \left| \delta_{\tau}(e^{-z\tau})-z  \right|
\leq c \tau^{k}|z|^{k+\gamma}.
\end{equation*}
The proof is completed.
\end{proof}

%\begin{lemma}[\cite{Cuesta:06}]\label{Lemma 2.4}
%There exist positive constants $c$ which depends only on  $\gamma$ such that
%\begin{equation*}
%||(z^{\gamma}+A)^{-1}||\le c|z|^{-\gamma} \quad {\rm and} \quad ||K(z)||\le c|z|^{-1-\gamma}.
%\end{equation*}
%\end{lemma}

\begin{lemma}\label{Lemma 2.5}%%%%[\cite{JLZ:17}]
Let $\delta_{\tau}$ and $K$ be given by \eqref{2.2} and \eqref{2.11}, respectively. Then there exist a  positive constants $c$ such that
\begin{equation*}
\begin{split}
%&\left\|K(\delta_{\tau}(e^{-z\tau}))-K(z)\right\|\le c\tau^{k}|z|^{k-1-\gamma}, \\
\left\|K(\delta_{\tau}(e^{-(\sigma+z)\tau}))-K(\sigma+z)\right\|\leq c\tau^{k}|\sigma+z|^{k-1-\gamma} \leq c\tau^{k}|z|^{k-1-\gamma},
\end{split}
\end{equation*}
where $\kappa$ in \eqref{2.a10} is large enough compared with $\sigma$.
\end{lemma}
\begin{proof}
Using  the triangle inequality and \eqref{2.11111}, we have
\begin{equation*}
\begin{split}
&\left\|K(\delta_{\tau}(e^{-(\sigma+z)\tau}))-K(\sigma+z)\right\|\\
=&\left\|(\delta_{\tau}(e^{-(\sigma+z)\tau}))^{-1}(\delta^{\gamma}_{\tau}(e^{-(\sigma+z)\tau})+A)^{-1}
-(\sigma+z)^{-1}\left((\sigma+z)^{\gamma}+A\right)^{-1}\right\|\\
\leq&\left|\left((\delta_{\tau}(e^{-(\sigma+z)\tau}))^{-1}-(\sigma+z)^{-1}\right)\right|\left\|(\delta^{\gamma}_{\tau}(e^{-(\sigma+z)\tau})+A)^{-1}\right\|\\
&+\left|(\sigma+z)^{-1}\right|\left\|\left(\delta^{\gamma}_{\tau}(e^{-(\sigma+z)\tau})+A\right)^{-1}-\left((\sigma+z)^{\gamma}+A\right)^{-1}\right\|\\
\leq& c\tau^{k}|\sigma+z|^{k-1-\gamma}.
\end{split}
\end{equation*}
Here we use \cite{JLZ:17}
\begin{equation*}
\begin{split}
\Big\|\left(\delta^{\gamma}_{\tau}(e^{-(\sigma+z)\tau})+A\right)^{-1}\Big\|\leq c|\sigma+z|^{-\gamma}
\end{split}
\end{equation*}
and
\begin{equation*}
\begin{split}
&\left(\delta^{\gamma}_{\tau}(e^{-(\sigma+z)\tau})+A\right)^{-1}-\left((\sigma+z)^{\gamma}+A\right)^{-1}\\
=&\left((\sigma+z)^{\gamma}-\delta^{\gamma}_{\tau}(e^{-(\sigma+z)\tau})\right)\left(\delta^{\gamma}_{\tau}(e^{-(\sigma+z)\tau})+A\right)^{-1}
\left((\sigma+z)^{\gamma}+A\right)^{-1}.
\end{split}
\end{equation*}
The proof is completed.
%since
%\begin{equation*}
%\begin{split}
%&\left|\left(\delta^{\gamma}_{\tau}(e^{-(\sigma+z)\tau})-A\right)^{-1}-\left((\sigma+z)^{\gamma}-A\right)^{-1}\right|\\
%=&\left|\left(\delta^{\gamma}_{\tau}(e^{-(\sigma+z)\tau})-(\sigma+z)^{\gamma}\right)
%\left(\delta^{\gamma}_{\tau}(e^{-(\sigma+z)\tau})-A\right)^{-1}\left((\sigma+z)^{\gamma}-A\right)^{-1}\right|\\
%\leq & c\tau^{k}|\sigma+z|^{k+\gamma}|\sigma+z|^{-2\gamma}
%\end{split}
%\end{equation*}
\end{proof}

%By comparing \eqref{2.10} and \eqref{2.13}, to obtain $\mathcal{O}(\tau^{k})$ accuracy, the following some conditions should be satisfied for $z\in \Gamma^{\tau}_{\theta,\kappa}$:
%\begin{equation*}
%\begin{split}
%& \left|\delta_{\tau}(e^{-(\sigma+z)\tau})-(\sigma+z)\right|\le c \tau^{k}\left|\sigma+z\right|^{k+1},
%  \quad \left|\mu_{1}(e^{-(\sigma+z)\tau})-1\right| \le c \tau^{k}\left|\sigma+z\right|^{k},\\
%& \left|\mu_{2}(e^{-z\tau})-1\right| \le c \tau^{k}\left|z\right|^{k},\\
%& \left|\frac{\delta_{\tau}(e^{-(\sigma+z)\tau})}{\delta_{\tau}(e^{-z\tau})}-\frac{\sigma+z}{z}\right|
%  \le c \tau^{k}\left(\left|\sigma+z\right|^{k+1}\left|z\right|^{-1}+\left|\sigma+z\right|\left|z\right|^{k-1}\right) , \\
%& \left| \left(\frac{\gamma_{l}(e^{-z\tau})}{l!} +\sum^{k-1}_{j=1}d^{(k)}_{l,n}e^{-z j\tau}\right)\tau^{l+1} - \frac{1}{z^{l+1}} \right|
%  \le c \tau^{k}\left|(\sigma+z)\right|^{k-l-1}.
%\end{split}
%\end{equation*}
%%By choice of $a^{(k)}_{j}$, $b^{(k)}_{j}$ and $d^{(k)}_{l,j}$,
%By choosing the suitable $a^{(k)}_{j}$, $b^{(k)}_{j}$ and $d^{(k)}_{l,j}$, such that the above inequality hold (see Table \ref{table:1} and Table \ref{table:2}). We were surprised to find that the choice of coefficients $a^{(k)}_{j}$, $b^{(k)}_{j}$ and $d^{(k)}_{l,j}$ were consistent with those of \cite{JLZ:17}. Let's write those inequality above as two lemmas.

\begin{lemma}\label{Lemma 2.0001}
Let $\delta_{\tau}$ be given by \eqref{2.2} with $1\le k \le 6$ and $\mu_{1}(\xi),\mu_{2}(\xi),\gamma_{l}(\xi)$ be given by \eqref{2.14}. Let $a^{(k)}_{j}$, $b^{(k)}_{j}$ and $d^{(k)}_{l,j}$ be given in  Table \ref{table:1} and Table \ref{table:2}. Then
\begin{equation*}
\begin{split}
& \left|\delta_{\tau}(e^{-(\sigma+z)\tau})-(\sigma+z)\right|\le c \tau^{k}\left|\sigma+z\right|^{k+1} \le c \tau^{k}\left|z\right|^{k+1},\\
& \left|\mu_{1}(e^{-(\sigma+z)\tau})-1\right| \le c \tau^{k}\left|\sigma+z\right|^{k} \le c \tau^{k}\left|z\right|^{k},\quad
  \left|\mu_{2}(e^{-z\tau})-1\right| \le c \tau^{k}\left|z\right|^{k},\\
& \left| \left(\frac{\gamma_{l}(e^{-z\tau})}{l!} +\sum^{k-1}_{j=1}d^{(k)}_{l,j}e^{-z j\tau}\right)\tau^{l+1} - \frac{1}{z^{l+1}} \right|
  \le c \tau^{k}\left|\sigma+z\right|^{k-l-1} \le c \tau^{k}\left|z\right|^{k-l-1},
\end{split}
\end{equation*}
where $\kappa$ in \eqref{2.a10} is large enough compared with $\sigma$.
\end{lemma}
\begin{proof}
From  Lemma \ref{Lemma 2.2}, infer that
\begin{equation*}
\left|\delta_{\tau}(e^{-(\sigma+z)\tau})-(\sigma+z)\right| \le c \tau^{k}\left|\sigma+z\right|^{k+1} \le c \tau^{k}\left|z\right|^{k+1}.
\end{equation*}
The others  inequality similar arguments can be performed as in \cite{JLZ:17}, we omit it here.
\end{proof}

\begin{lemma}\label{Lemma 2.0002}
Let $\delta_{\tau}(\xi)$ be  given by \eqref{2.2} with $1\le k \le 6$. Then
\begin{equation*}
\begin{split}
\left|\frac{\delta_{\tau}(e^{-(\sigma+z)\tau})}{\delta_{\tau}(e^{-z\tau})}-\frac{\sigma+z}{z}\right|
  \le c \tau^{k}\left(\left|\sigma+z\right|^{k+1}\left|z\right|^{-1}+\left|\sigma+z\right|\left|z\right|^{k-1}\right)\le c \tau^{k}\left|z\right|^{k},
\end{split}
\end{equation*}
where $\kappa$ in \eqref{2.a10} is large enough compared with $\sigma$.
\end{lemma}
\begin{proof}
According to Lemma \ref{Lemma 2.3} and \ref{Lemma 2.0001}, we have
%\begin{equation*}
%\begin{split}
% & \left|\frac{\delta_{\tau}(e^{-(\sigma+z)\tau})}{\delta_{\tau}(e^{-z\tau})}-\frac{\sigma+z}{z}\right|
%  \le \left|\frac{(\sigma+z)\tau+c((\sigma+z)\tau)^{k+1}}{z\tau-c(z\tau)^{k+1}}-\frac{(\sigma+z)\tau}{z\tau}\right|
%   =    \left|\frac{(\sigma+z)\tau}{z\tau}\right|\left|\frac{c((\sigma+z)\tau)^{k}+c(z\tau)^{k}}{1-c(z\tau)^{k}}\right|\\
% & \le \left|\frac{(\sigma+z)\tau}{z\tau}\right|\left|c((\sigma+z)\tau)^{k}+c(z\tau)^{k}\right|
%  \le   c \tau^{k}\left(\left|\sigma+z\right|^{k+1}\left|z\right|^{-1}+\left|\sigma+z\right|\left|z\right|^{k-1}\right)\le c \tau^{k}\left|z\right|^{k}.
%\end{split}
%\end{equation*}
\begin{equation*}
\begin{split}
 & \left|\frac{\delta_{\tau}(e^{-(\sigma+z)\tau})}{\delta_{\tau}(e^{-z\tau})}-\frac{\sigma+z}{z}\right|
%   = \left|\frac{z\delta_{\tau}(e^{-(\sigma+z)\tau})-(\sigma+z)\delta_{\tau}(e^{-z\tau})}{z\delta_{\tau}(e^{-z\tau})}\right|
   = \left|\frac{z\delta_{\tau}(e^{-(\sigma+z)\tau})-z(\sigma+z)+(\sigma+z)z-(\sigma+z)\delta_{\tau}(e^{-z\tau})}{z\delta_{\tau}(e^{-z\tau})}\right|\\
&   \le \left|\frac{\delta_{\tau}(e^{-(\sigma+z)\tau})-(\sigma+z)}{\delta_{\tau}(e^{-z\tau})}\right|
   +\left|\frac{(\sigma+z)\left(z-\delta_{\tau}(e^{-z\tau})\right)}{z\delta_{\tau}(e^{-z\tau})}\right|\\
&   \le\frac{c \tau^{k}\left|\sigma+z\right|^{k+1}}{\left|\delta_{\tau}(e^{-z\tau})\right|}
   +\frac{c \tau^{k}\left|\sigma+z\right|\left|z\right|^{k}}{\left|\delta_{\tau}(e^{-z\tau})\right|} \le\frac{c \tau^{k}\left|\sigma+z\right|^{k+1}+c \tau^{k}\left|\sigma+z\right|\left|z\right|^{k}}{c_{1}\left|z\right|}\\
& \le c \tau^{k}\left(\left|\sigma+z\right|^{k+1}\left|z\right|^{-1}+\left|\sigma+z\right|\left|z\right|^{k-1}\right)\le c \tau^{k}\left|z\right|^{k}.
\end{split}
\end{equation*}
The proof is completed.
\end{proof}

%%%%%%%%%%%%%%%%%%%%%%%%%%%%%%%%%%%%%%%%%%%%%%%%%%%%%%%%%%%%%%%%%%%%%%%%%%%%%%%%%%%%%%%%%%%%%%%%%%%%%%%%%%%%%%%%%%%%%%%%%%%%%%%%%

\begin{lemma}\label{lemma3.1}
Let $\gamma_{l}(\xi)$ be given by \eqref{2.14}. Then
\begin{equation*}
\left| \frac{\gamma_{l}(e^{-z\tau})}{l!}\tau^{l+1}-\frac{1}{z^{l+1}}\right|\leq c\tau^{l+1}\quad l=1,2,\ldots,5,\quad z\in\Gamma^{\tau}_{\theta,\kappa}.
\end{equation*}
\end{lemma}
\begin{proof}
From \eqref{2.14}, we have
\begin{equation*}
\begin{split}
\gamma_{1}(e^{-z\tau})&=\frac{e^{-z\tau}}{\left(1-e^{-z\tau}\right)^2},\qquad
\gamma_{2}(e^{-z\tau})=\frac{e^{-z\tau}+e^{-2z\tau}}{\left(1-e^{-z\tau}\right)^3},\\
\gamma_{3}(e^{-z\tau})&=\frac{e^{-z\tau}+4e^{-2z\tau}+e^{-3z\tau}}{\left(1-e^{-z\tau}\right)^4},\\
\gamma_{4}(e^{-z\tau})&=\frac{e^{-z\tau}+11e^{-2z\tau}+11e^{-3z\tau}+e^{-4z\tau}}{\left(1-e^{-z\tau}\right)^5},\\
\gamma_{5}(e^{-z\tau})&=\frac{e^{-z\tau}+26e^{-2z\tau}+66e^{-3z\tau}+26e^{-4z\tau}+e^{-5z\tau}}{\left(1-e^{-z\tau}\right)^6}.
\end{split}
\end{equation*}
Using  the Taylor series expansion, we have
\begin{equation}\label{3.0001}
\left| (1-e^{-z\tau})^{l+1}\tau^{-(l+1)}z^{l+1} \right|\geq c|z|^{2l+2}.
\end{equation}
For simplicity, we denote $ \gamma_{l}(e^{-z\tau}) = \frac{\psi_{l}(e^{-z\tau})}{\rho_{l}(e^{-z\tau})}$ with $\rho_{l}(e^{-z\tau})=(1-e^{-z\tau})^{l+1}$, it yields
\begin{equation*}
\begin{split}
&\left|\left(\psi_{l}(e^{-z\tau})\right)z^{l+1}- l!\rho_{l}(e^{-z\tau})\tau^{-(l+1)}\right|\\
=&\left| \tau^{l+2}z^{2l+3}\sum^{\infty}_{n=0}\left( \frac{\sum^{l}_{j=1}p_{l,j}(-jz\tau)^{n}}{(n+l+1)!}
-l!\frac{\sum^{l+1}_{j=1}c_{l,j}(-jz\tau)^{n}}{(n+2l+2)!}\right)\right|\\
\leq& c \tau^{l+2}|z|^{2l+3} \leq c \tau^{l+1}|z|^{2l+2} \quad {\rm if}~~l=2,4,
\end{split}
\end{equation*}
and
\begin{equation*}
\begin{split}
&\left|\left(\psi_{l}(e^{-z\tau})\right)z^{l+1}- l!\rho_{l}(e^{-z\tau})\tau^{-(l+1)}\right|\\
=&\left| \tau^{l+1}z^{2l+2}\sum^{\infty}_{n=0}\left( \frac{\sum^{l}_{j=1}p_{l,j}(-jz\tau)^{n}}{(n+l+1)!}
-l!\frac{\sum^{l+1}_{j=1}c_{l,j}(-jz\tau)^{n}}{(n+2l+2)!}\right)\right|\\
\leq& c \tau^{l+1}|z|^{2l+2} \quad {\rm if}~~l=1,3,5.
\end{split}
\end{equation*}
Here the coefficients $p_{l,j}$ and $c_{l,j}$ are, respectively,  given in Table \ref{table:3} and Table \ref{table:4}.
\begin{table}[h]\fontsize{9pt}{12pt}\selectfont%Éú³É¸¡¶¯±í¸ñ
 \begin{center}%\def\tabcolsep{28.5pt}%±í¸ñ¾ÓÖÐ
  \caption {The coefficients $p_{l,j}$.} \vspace{5pt}%±êÌ⣬Àë±í¸ñÒ»¶¨µÄ¾àÀë
\begin{tabular}{|c| c c c c c |}                                \hline  %»­¶¥¶ËµÄºáÏß
                &  $p_{l,1}$  &  $p_{l,2}$  &  $p_{l,3}$  &  $p_{l,4}$  &  $p_{l,5}$       \\ \hline
 $l=1$          &  1          &             &             &             &                  \\ \hline
 $l=2$          &  1          &  16         &             &             &                  \\ \hline
 $l=3$          &  1          &  64         &  8          &             &                  \\ \hline
 $l=4$          &  1          &  704        &  8019       &  4096       &                  \\ \hline
 $l=5$          &  1          &  1664       &  48114      &  106496     &  15625           \\ \hline% »­µ×¶ËµÄºáÏß
%%%%%%%%%%%%%%%%%%%%%%%%%%%%%%%%%%%%%%%%%%%%%%%%%%%%%%%%%%%%%%%%%%%%%%%%%%%%%%%%%%%%%%%%%%%%%%%%%%%%%%%%%%%%%%%%%%%%%%%%%%%%%%%%%%%
    \end{tabular}\label{table:3}%\vspace{-15pt}
  \end{center}
\end{table}
\begin{table}[h]\fontsize{9pt}{12pt}\selectfont%Éú³É¸¡¶¯±í¸ñ
 \begin{center}%\def\tabcolsep{28.5pt}%±í¸ñ¾ÓÖÐ
  \caption {The coefficients $c_{l,j}$.} \vspace{5pt}%±êÌ⣬Àë±í¸ñÒ»¶¨µÄ¾àÀë
\begin{tabular}{|c| c c c c c c|}                                \hline  %»­¶¥¶ËµÄºáÏß
               &  $c_{l,1}$  &  $c_{l,2}$  &  $c_{l,3}$  &  $c_{l,4}$  &  $c_{l,5}$  &  $c_{l,6}$  \\ \hline
 $l=1$         &  -2         &  16         &             &             &             &             \\ \hline
 $l=2$         &  3          &  -384       & 2187        &             &             &             \\ \hline
 $l=3$         &  -4         &  1536       & -26244      & 65536       &             &             \\ \hline
 $l=4$         &  5          &  -20480     & 1771470     & -209715201  &    48828125 &             \\ \hline
 $l=5$         &  -6         &  61440      & -10628820   & 251658240   & -1464843750 & 2176782336  \\ \hline% »­µ×¶ËµÄºáÏß
%%%%%%%%%%%%%%%%%%%%%%%%%%%%%%%%%%%%%%%%%%%%%%%%%%%%%%%%%%%%%%%%%%%%%%%%%%%%%%%%%%%%%%%%%%%%%%%%%%%%%%%%%%%%%%%%%%%%%%%%%%%%%%%%%%%
    \end{tabular}\label{table:4}%\vspace{-15pt}
  \end{center}
\end{table}
The proof is completed.
\end{proof}

\subsection{Error analysis}
We now given the error analysis of correction BDF$k$ \eqref{2.4} for \eqref{1.1}.
\begin{lemma}\label{lemma3.2}
Let $G(t)$ and $G^{n}$ be the solutions of \eqref{1.1}  and \eqref{2.4}, respectively.
If $G_{0}=0$ and $f(t)=\frac{t^{k-1}}{(k-1)!} g$ with $g=\partial^{k-1}_{t}f(0)$, then
\begin{equation*}
\left\| G^{n}-G(t_{n}) \right\| \leq c \tau^{k}\int^{t_{n}}_{0} (t_{n}-s)^{\gamma-1} \| g \| ds.
\end{equation*}
\end{lemma}
\begin{proof}
From \eqref{2.10} and \eqref{2.13}, we have
\begin{equation*}
G(t_{n})=\frac{1}{2\pi i}\int_{\Gamma_{\theta,\kappa}}{e^{zt_{n}}((\sigma+z)^{\gamma}+A)^{-1}\frac{1}{z^{k}}g}dz
\end{equation*}
and
\begin{equation*}
G^{n}=\frac{1}{2\pi i}\int_{\Gamma^{\tau}_{\theta,\kappa}}e^{zt_{n}}(\delta_{\tau}^{\gamma}(e^{-(\sigma+z)\tau})+A)^{-1}
\frac{\gamma_{k-1}(e^{-z\tau})}{(k-1)!}\tau^{k} g dz.
\end{equation*}
Then
\begin{equation*}
\begin{split}
&G(t_{n})-G^{n}=\textrm{\uppercase\expandafter{\romannumeral1}} + \textrm{\uppercase\expandafter{\romannumeral2}}. %%%%%Èç¹û²»¼Ó\textrm{}£¬ÂÞÂíÊý×Ö»á±ä³ÉбÌå¡£
\end{split}
\end{equation*}
Here

\begin{equation*}
\begin{split}
\textrm{\uppercase\expandafter{\romannumeral1}}
=\frac{1}{2\pi i}\int_{\Gamma^{\tau}_{\theta,\kappa}}e^{zt_{n}}\left(((\sigma+z)^{\gamma}+A)^{-1}\frac{1}{z^{k}}
-(\delta_{\tau}^{\gamma}(e^{-(\sigma+z)\tau})+A)^{-1}\frac{\gamma_{k-1}(e^{-z\tau})}{(k-1)!}\tau^{k}\right) g dz
\end{split}
\end{equation*}
and
\begin{equation*}
\begin{split}
\textrm{\uppercase\expandafter{\romannumeral2}}
=\frac{1}{2\pi i}\int_{\Gamma_{\theta,\kappa}\setminus\Gamma^{\tau}_{\theta,\kappa}}{e^{zt_{n}}((\sigma+z)^{\gamma}+A)^{-1}z^{-k}g}dz
\end{split}
\end{equation*}
with
\begin{equation*}
\Gamma_{\theta,\kappa} \backslash \Gamma^{\tau}_{\theta,\kappa}=
\left\{z\in\mathbb{C}:z=re^{\pm i\theta},\frac{\pi}{\tau\sin(\theta)}\le r<\infty \right\}.
\end{equation*}
Using Lemma \ref{lemma3.1}, it yields
%\begin{equation*}
%\| \textrm{\uppercase\expandafter{\romannumeral1}} \|
%\leq c\tau^{k}\|g\|\left( \int^{\frac{\pi}{\tau\sin\theta}}_{\kappa} e^{rt_{n}\cos\theta} r^{-\gamma}dr +\int^{\theta}_{-\theta}e^{\kappa t_{n}\cos\psi} \kappa^{1-\gamma} d\psi \right)
%\leq  c\tau^{k}t_{n}^{\gamma-1}\|g\|
%\end{equation*}
\begin{equation*}
\begin{split}
  \| \textrm{\uppercase\expandafter{\romannumeral1}} \|
  &  \leq c\tau^{k}\|g\|\left( \int^{\frac{\pi}{\tau\sin\theta}}_{\kappa} e^{rt_{n}\cos\theta} r^{-\gamma}dr +\int^{\theta}_{-\theta}e^{\kappa t_{n}\cos\psi} \kappa^{1-\gamma} d\psi \right)\\
  &  =  c\tau^{k}\|g\|\left(t_n^{\gamma-1} \int^{\frac{t_n\pi}{\tau\sin\theta}}_{t_n\kappa} e^{s\cos\theta} s^{-\gamma}dr +\kappa^{1-\gamma}\int^{\theta}_{-\theta}e^{\kappa t_{n}\cos\psi}  d\psi \right)\\
  &  \leq c\tau^{k}\|g\|\left(t_n^{\gamma-1} +\kappa^{1-\gamma}\int^{\theta}_{-\theta}e^{\kappa T}  d\psi \right)
    \leq c\tau^{k}\|g\|\left(t_n^{\gamma-1} +\kappa^{1-\gamma}e^{\kappa T}  \right)\\
  &  \leq c\tau^{k}\|g\|\left(t_n^{\gamma-1} +(\kappa T)^{1-\gamma}e^{\kappa T}t_n^{\gamma-1}  \right)
    \leq  c\tau^{k}t_{n}^{\gamma-1}\|g\|
\end{split}
\end{equation*}
and
\begin{equation*}
\begin{split}
\| \textrm{\uppercase\expandafter{\romannumeral2}} \|
&\leq c \|g\| \int^{\infty}_{\frac{\pi}{\tau\sin\theta}} e^{rt_{n}\cos\theta}r^{-k-\gamma}dr\\
&\leq c\tau^{k} \|g\| \int^{\infty}_{\frac{\pi}{\tau\sin\theta}} e^{rt_{n}\cos\theta}r^{-\gamma}dr
\leq  c\tau^{k}t_{n}^{\gamma-1}\|g\|.
\end{split}
\end{equation*}
According to the triangle inequality,  the desired result are obtained.
\end{proof}

\begin{lemma}\label{lemma3.3}
Let $G(t)$ and $G^{n}$ be the solutions of \eqref{1.1}  and \eqref{2.4}, respectively.
If $G_{0}=0$ and $f(t)=\frac{t^{k-1}}{(k-1)!}\ast g(t)$, then
\begin{equation*}
\left\| G^{n}-G(t_{n}) \right\| \leq c \tau^{k}\int^{t_{n}}_{0} (t_{n}-s)^{\gamma-1} \| g(s) \| ds.
\end{equation*}
\end{lemma}
\begin{proof}
From  \eqref{2.10}, we have
\begin{equation*}
\begin{split}
G(t_{n})
&=\frac{1}{2\pi i}\int_{\Gamma_{\theta,\kappa}}{e^{zt_{n}}((\sigma+z)^{\gamma}+A)^{-1}\widehat{f}(z)}dz\\
&=\left(\frac{1}{2\pi i}\int_{\Gamma_{\theta,\kappa}}{e^{zt}((\sigma+z)^{\gamma}+A)^{-1}}dz\ast f\right)(t_{n})\\
&=(\mathscr{E}\ast f)(t_{n})=\left(\mathscr{E}\ast \left(\frac{t^{k-1}}{(k-1)!}\ast g(t)\right)\right)(t_{n})
=\left(\left(\mathscr{E}\ast\frac{t^{k-1}}{(k-1)!}\right)\ast g(t)\right)(t_{n})
\end{split}
\end{equation*}
with $\mathscr{E}(t)= \frac{1}{2\pi i}\int_{\Gamma_{\theta,\kappa}}{e^{zt}((\sigma+z)^{\gamma}+A)^{-1}}dz$.

 Using the generating function $ \widetilde{f}(\xi)=\sum^{\infty}_{n=0}f(t_{n})\xi^{n} $ and
 $$\widetilde{G}(\xi)
=\left( \delta^{\gamma}_{\tau}\left(e^{-\sigma\tau}\xi\right) +A\right)^{-1}\widetilde{f}(\xi)
:=\widetilde{\mathscr{E}}(\delta_{\tau}(\xi))\widetilde{f}(\xi) $$ in   \eqref{5.4}, we obtain
\begin{equation*}
G^{n}=\sum^{n}_{j=0}\mathscr{E}^{n}_{\tau}f(t_{j}) \quad{\rm with} \quad \widetilde{\mathscr{E}}(\delta_{\tau}(\xi))=\sum^{\infty}_{n=0}\mathscr{E}^{n}_{\tau}\xi^{n}.
\end{equation*}

From Cauchy's integral formula and taking the change of variables $\xi=e^{-z\tau}$, we have following integral representation
\begin{equation*}
\mathscr{E}^{n}_{\tau}=\frac{\tau}{2\pi i}\int_{\Gamma^{\tau}_{\theta,\kappa}}{e^{zn\tau}\left(\delta^{\gamma}_{\tau}\left(e^{-(\sigma+z)\tau}\right) +A\right)^{-1}}dz.
\end{equation*}
Using Lemma \ref{Lemma 2.3}, it means that
\begin{equation}\label{3.0002}
\|\mathscr{E}^{n}_{\tau}\| \leq c \tau \left( \int^{\frac{\pi}{\tau\sin(\theta)}}_{\kappa} e^{rt_{n}\cos(\theta)} r^{-\gamma}dr +\int^{\theta}_{-\theta}e^{\kappa t_{n}\cos(\psi)} \kappa^{-\gamma} \kappa d\psi\right)
\leq c\tau t_{n}^{\gamma-1}.
\end{equation}

Let $ \mathscr{E}_{\tau}(t)=\sum^{\infty}_{n=0}\mathscr{E}^{n}_{\tau}\delta_{t_{n}}(t)$ with $\delta_{t_{n}}$  the Dirac delta function at $t_{n}$.
Then
\begin{equation*}
\begin{split}
  (\mathscr{E}_{\tau}(t)\ast f(t))(t_{n})
  &=\left(\sum^{\infty}_{j=0}\mathscr{E}^{j}_{\tau}\delta_{t_{j}}(t) \ast f(t) \right)(t_{n})
  =\sum^{n}_{j=0}\mathscr{E}^{j}_{\tau}f(t_{n}-t_{j})\\
  &=\sum^{n}_{j=0}\mathscr{E}^{n-j}_{\tau}f(t_{j})=G^{n}.
\end{split}
\end{equation*}
Moreover, we have
\begin{equation*}
\begin{split}
  \widetilde{(\mathscr{E}_{\tau}\ast t^{k-1})}(\xi)
& = \sum^{\infty}_{n=0} \sum^{n}_{j=0}\mathscr{E}^{n-j}_{\tau}t_{j}^{k-1}\xi^{n}
  =\sum^{\infty}_{j=0} \sum^{\infty}_{n=j}\mathscr{E}^{n-j}_{\tau}t_{j}^{k-1}\xi^{n}\\
& =\sum^{\infty}_{j=0} \sum^{\infty}_{n=0}\mathscr{E}^{n}_{\tau}t_{j}^{k-1}\xi^{n+j}
  =\sum^{\infty}_{n=0}\mathscr{E}^{n}_{\tau}\xi^{n}\sum^{\infty}_{j=0}t_{j}^{k-1}\xi^{j}\\
& =\widetilde{\mathscr{E}}(\delta_{\tau}(\xi)) \tau^{k-1}\sum^{\infty}_{j=0}j^{k-1}\xi^{j}
  =\widetilde{\mathscr{E}}(\delta_{\tau}(\xi)) \tau^{k-1}\gamma_{k-1}(\xi).
\end{split}
\end{equation*}
From Lemma \ref{lemma3.2}, we have the following estimate
\begin{equation*}
\left\|\left((\mathscr{E}_{\tau}-\mathscr{E}) \ast \frac{t^{k-1}}{(k-1)!}\right)(t_{n})\right\| \leq c\tau^{k}t_{n}^{\gamma-1}.
\end{equation*}

Next, we prove the following inequality \eqref{3.0003}  for $t>0$
\begin{equation}\label{3.0003}
\left\|\left((\mathscr{E}_{\tau}-\mathscr{E}) \ast \frac{t^{k-1}}{(k-1)!}\right)(t)\right\| \leq c\tau^{k}t^{\gamma-1},\quad \forall t\in (t_{n-1},t_{n}).
\end{equation}
Using the Taylor series expansion of $\mathscr{E}(t)$ at $t=t_{n}$, we get
\begin{equation*}
\begin{split}
 &\left( \mathscr{E} \ast \frac{t^{k-1}}{(k-1)!}\right)(t)\\
=&\left( \mathscr{E} \ast \frac{t^{k-1}}{(k-1)!}\right)(t_{n})+(t-t_{n})\left( \mathscr{E} \ast \frac{t^{k-2}}{(k-2)!}\right)(t_{n})+\cdots
 + \frac{(t-t_{n})^{k-2}}{(k-2)!}\left( \mathscr{E} \ast t \right)(t_{n})\\
 &+ \frac{(t-t_{n})^{k-1}}{(k-1)!}\left( \mathscr{E} \ast 1 \right)(t_{n})
 +\frac{1}{(k-1)!}\int^{t}_{t_{n}}(t-s)^{k-1}\mathscr{E}(s)ds.
\end{split}
\end{equation*}
This above  expansion also holds  for $  \left( \mathscr{E}_{\tau} \ast \frac{t^{k-1}}{(k-1)!}\right)(t) $.
Then we have
\begin{equation*}
\left\|\left((\mathscr{E}_{\tau}-\mathscr{E}) \ast \frac{t^{l}}{l!}\right)(t_{n})\right\| \leq c\tau^{l+1}t_{n}^{\gamma-1}.
\end{equation*}

Using \eqref{2.11111}, we have
\begin{equation*}
\begin{split}
\| \mathscr{E}(t) \|
\leq c \left( \int^{\infty}_{\kappa}e^{rt\cos\theta}r^{-\gamma}dr + \int^{\theta}_{-\theta}e^{\kappa t \cos\psi}\kappa^{1-\gamma}d\psi \right)
\leq c t^{\gamma-1}
\end{split}
\end{equation*}
and
\begin{equation*}
\left\| \int^{t}_{t_{n}}(t-s)^{k-1}\mathscr{E}(s)ds \right\| \leq c \int^{t_{n}}_{t}(s-t)^{k-1}s^{\gamma-1}ds \leq c \tau^{k} t^{\gamma-1}.
\end{equation*}
Similarly, from \eqref{3.0002}, we deduce
\begin{equation*}
\left\| \int^{t}_{t_{n}}(t-s)^{k-1}\mathscr{E}_{\tau}(s)ds \right\| \leq c\tau^{k-1} \| \mathscr{E}^{n}_{\tau} \| \leq c \tau^{k} t_{n}^{\gamma-1}.
\end{equation*}
Then we can obtain \eqref{3.0003} by $t_{n}^{\gamma-1}\leq t^{\gamma-1}$ for $t\in (t_{n-1},t_{n})$ and $\gamma\in(0,1)$.
The proof is completed.
\end{proof}

\begin{theorem}\label{Theorem 1}
Let $f\in C^{k-1}([0,T];L^{2}(\Omega))$ and $\int^{t}_{0}{(t-s)^{\gamma-1}}||\partial^{l}_{s}f(s)||_{L^{2}(\Omega)}ds<\infty$. Let $G(t_{n})$ and $G^{n}$ be the  solutions of \eqref{1.1} and \eqref{2.4} at the point $t_{n}$, respectively. Let $\varepsilon^{n}=G^{n}-G(t_{n})$ with $\varepsilon^{0}=0$. Then
\begin{equation*}
\begin{split}
||\varepsilon^{n}||
=||G^{n}-G(t_{n})||
&\leq c\tau^{k} \left(t^{\gamma-k}_{n}\left\|A G^{0}\right\| + t^{\gamma-k}_{n}\left\|f(0)\right\|
    +\sum^{k-1}_{l=1}t^{\gamma+l-k}_{n}\left\|\partial^{l}_{t}f(0)\right\| \right. \\
&   \left. \qquad\qquad + \int^{t_{n}}_{0} (t_{n}-s)^{\gamma-1} \| \partial^{k}_{s}f(s) \| ds\right).
\end{split}
\end{equation*}
\end{theorem}

\begin{proof}
Using
$$G^{n}-G(t_{n})=W^{n}+e^{-\sigma n \tau}G(0)-(W(t_{n})+e^{-\sigma t_{n}}G(0))=W^{n}-W(t_{n})$$ and subtracting \eqref{2.10} from \eqref{2.13}, we have
\begin{equation}\label{3.1}
G^{n}-G(t_{n})=I_{1}+I_{2}+\sum^{k-2}_{l=1}I_{l,3}+I_{4}-I_{5}.
\end{equation}
Here
\begin{equation*}
\begin{split}
I_{1} =& \frac{1}{2 \pi i} \int_{\Gamma^{\tau}_{\theta,\kappa}} -e^{t_{n}z} \left[K\left(\delta_{\tau}(e^{-(\sigma+z)\tau})\right)
         \mu_{1}(e^{-(\sigma+z)\tau}) -K(\sigma+z)\right]AG^{0}dz; \\
I_{2} =& \frac{1}{2 \pi i} \!\int_{\Gamma^{\tau}_{\theta,\kappa}}\!\! e^{t_{n}z} \left[K\left(\delta_{\tau}(e^{-(\sigma+z)\tau})\right)
         \frac{\delta_{\tau}(e^{-(\sigma+z)\tau})}{\delta_{\tau}(e^{-z\tau})}\mu_{2}(e^{-z\tau}) \!-\!K(\sigma+z)\frac{\sigma+z}{z}\right]f(0)dz ;\\
I_{l,3} =& \frac{1}{2 \pi i}\!\int_{\Gamma^{\tau}_{\theta,\kappa}}\!\!e^{t_{n}z}\!\left[K\left(\delta_{\tau}(e^{-(\sigma+z)\tau})\right)\delta_{\tau}
(e^{-(\sigma+z)\tau})\!\sum^{k-2}_{l=1}\!\left(\frac{\gamma_{l}(e^{-z\tau})}{l!} \right. \right. \\
&\qquad\qquad\qquad\qquad \qquad \left. \left.+\sum^{k-1}_{j=1}\!d^{(k)}_{l,n}e^{-z j\tau}\right)\!\tau^{l+1}
 -K(\sigma+z)\frac{\sigma+z}{z^{l+1}}\right]\partial^{l}_{t}f(0)dz ;\\
I_{4} =& \frac{1}{2 \pi i} \int_{\Gamma^{\tau}_{\theta,\kappa}}e^{t_{n}z}K\left(\delta_{\tau}(e^{-(\sigma+z)\tau})\right)
  \delta_{\tau}(e^{-(\sigma+z)\tau})\tau\widetilde{R}_{k}(e^{-z\tau})dz\\
 & - \frac{1}{2\pi i} \int_{\Gamma_{\theta,\kappa}} {e^{zt_{n}}K(\sigma+z)(\sigma+z)\widehat{R}_{k} } dz;\\
I_{5} =& \frac{1}{2\pi i} \int_{\Gamma_{\theta,\kappa}\backslash \Gamma^{\tau}_{\theta,\kappa}} {e^{zt_{n}}K(\sigma+z)\left( -A G(0)+\frac{\sigma+z}{z} f(0) + (\sigma+z)\sum^{k-2}_{l=1}\frac{1}{z^{l+1}}\partial^{l}_{t}f(0) \right) } dz.
\end{split}
\end{equation*}
According to Lemma \ref{Lemma 2.0001} and Lemma \ref{Lemma 2.5}, we estimate the first term $I_{1}$ as following
\begin{equation*}
\begin{split}
\left\|I_{1}\right\|
=& \left\|\frac{1}{2 \pi i} \int_{\Gamma^{\tau}_{\theta,\kappa}} -e^{t_{n}z} \left[K\left(\delta_{\tau}(e^{-(\sigma+z)\tau})\right)
         \mu_{1}(e^{-(\sigma+z)\tau}) -K(\sigma+z)\right]AG^{0}dz \right\|\\
%\le& c\tau^{k} \left\| AG^{0} \right\|  \int_{\Gamma^{\tau}_{\theta,\kappa}} e^{t_{n}|z|\cos(\arg(z))} \left|dz\right| \\
\le& c\tau^{k} \left\| AG^{0} \right\|  \int^{\frac{\pi}{\tau\sin\theta}}_{\kappa} e^{t_{n}r\cos\theta} r^{k-1-\gamma}dr
    +c\tau^{k} \left\| AG^{0} \right\|  \int^{\theta}_{-\theta} e^{t_{n}\kappa\cos\psi} \kappa^{k-\gamma} d\psi\\
\le& c\tau^{k} \left\| AG^{0} \right\|  t_{n}^{\gamma-k}\int^{\frac{\pi t_{n}}{\tau\sin\theta}}_{\kappa t_{n}} e^{s\cos\theta} ds
    +c\tau^{k} \left\| AG^{0} \right\|  \kappa^{k-\gamma}\int^{\theta}_{-\theta} e^{T\kappa} d\psi\\
\le& c\tau^{k} \left\| AG^{0} \right\|  (t_{n}^{\gamma-k} + \kappa^{k-\gamma} e^{T\kappa} )
\le  c\tau^{k} \left\|AG^{0}  \right\|  \left(t_{n}^{\gamma-k}+(T\kappa)^{k-\gamma}e^{T\kappa}t_{n}^{\gamma-k}\right)\\
\le& c\tau^{k}\left\|AG^{0}\right\|t_{n}^{\gamma-k}.
\end{split}
\end{equation*}
From Lemma \ref{Lemma 2.0001}, Lemma \ref{Lemma 2.0002} and Lemma \ref{Lemma 2.5}, we estimate the second term $I_{2}$ as following in a similar way to $I_{1}$, i.e.,
\begin{equation*}
\left\|I_{2}\right\|  \le c\tau^{k} \left\| f(0) \right\|  t_{n}^{\gamma-k}.
\end{equation*}
By Lemma \ref{Lemma 2.0001} and Lemma \ref{Lemma 2.5}, we estimate the third term $I_{l,3}$
\begin{equation*}
\begin{split}
\left\|I_{l,3}\right\|
%=& \frac{1}{2 \pi i}\int_{\Gamma^{\tau}_{\theta,\kappa}}e^{t_{n}z}\left[K\left(\delta_{\tau}(e^{-(\sigma+z)\tau})\right)\delta_{\tau}
% (e^{-(\sigma+z)\tau})\sum^{k-2}_{l=1}\left(\frac{\gamma_{l}(e^{-z\tau})}{l!} +\sum^{k-1}_{j=1}d^{(k)}_{l,n}e^{-z j\tau}\right)\tau^{l+1} \right.\\
% &\qquad\qquad\qquad\quad  \left.-K(\sigma+z)\frac{\sigma+z}{z^{l+1}}\right]\partial^{l}_{t}f(0)dz \\
\le&   c\tau^{k} \left\| \partial^{l}_{t}f(0)\right\| \left(\int^{\frac{\pi}{\tau\sin\theta}}_{\kappa} e^{t_{n}r\cos\theta}r^{k-l-1-\gamma} dr
    + \int^{\theta}_{-\theta} e^{t_{n}\kappa\cos\psi} \kappa^{k-l-\gamma} d\psi \right)\\
\le  &   c\tau^{k} \left\| \partial^{l}_{t}f(0)\right\| t_{n}^{\gamma+l-k}, \qquad l=1,2,\ldots,k-2.
\end{split}
\end{equation*}
Direct calculation $I_{5}$ as following
\begin{equation*}
\begin{split}
\left\|I_{5}\right\|
%\le& \frac{1}{2\pi i} \int_{\Gamma_{\theta,\kappa}\backslash \Gamma_{\theta,\kappa}} {e^{zt}K(\sigma+z)\left( A G(0)+\frac{\sigma+z}{z} f(0) + (\sigma+z)\sum^{k-2}_{l=1}\frac{1}{z^{l+1}}\partial^{l}_{t}f(0) \right) } dz\\
\le& c\int_{\Gamma_{\theta,\kappa}\backslash \Gamma_{\theta,\kappa}} {\left|e^{zt_{n}}\right||\sigma+z|^{-1-\gamma} \left\|A G(0)\right\| } |dz|\\
&+c\int_{\Gamma_{\theta,\kappa}\backslash \Gamma_{\theta,\kappa}} {\left|e^{zt_{n}}\right||\sigma+z|^{-1-\gamma}
          \left|\frac{\sigma+z}{z}\right| \left\|f(0)\right\| } |dz|\\
&+c\int_{\Gamma_{\theta,\kappa}\backslash \Gamma_{\theta,\kappa}} {\left|e^{zt_{n}}\right||\sigma+z|^{-1-\gamma}|\sigma+z|\sum^{k-2}_{l=1}
          \frac{1}{|z|^{l+1}}\left\|\partial^{l}_{t}f(0)\right\|} |dz|\\
\le& c\left\|A G(0)\right\|\int^{\infty}_{\frac{\pi}{\tau\sin\theta}} {e^{t_{n}r\cos\theta}r^{-1-\gamma} } dr
    +c\left\|f(0)\right\|\int^{\infty}_{\frac{\pi}{\tau\sin\theta}} {e^{t_{n}r\cos\theta}r^{-1-\gamma}  } dr\\
&   +c\sum^{k-2}_{l=1}\left\|\partial^{l}_{t}f(0)\right\|\int^{\infty}_{\frac{\pi}{\tau\sin\theta}} {e^{t_{n}r\cos\theta}\frac{1}{r^{l+1+\gamma}}} dr\\
\le& c\tau^{k} \left(t^{\gamma-k}_{n}\left\|A G(0)\right\| + t^{\gamma-k}_{n}\left\|f(0)\right\|
    +\sum^{k-2}_{l=1}t^{\gamma+l-k}_{n}\left\|\partial^{l}_{t}f(0)\right\|\right)
\end{split}
\end{equation*}
for the last inequation, we use
\begin{equation*}
\begin{split}
\int^{\infty}_{\frac{\pi}{\tau\sin\theta}} {e^{t_{n}r\cos\theta}r^{-1-\gamma} } dr
\le& c\tau^{k} \int^{\infty}_{\frac{\pi}{\tau\sin\theta}} {e^{t_{n}r\cos\theta}r^{k-1-\gamma} } dr\\
=&    c\tau^{k} t^{\gamma-k}_{n}\int^{\infty}_{\frac{t_{n}\pi}{\tau\sin\theta}} {e^{s\cos\theta}s^{k-1-\gamma} } ds \le c\tau^{k} t^{\gamma-k}_{n},
\end{split}
\end{equation*}
for the above inequality, we using $1\le \left( \frac{\sin\theta}{\pi} \right)^{k} \tau^{k}r^{k}$, since $r\ge\frac{\pi}{\tau\sin\theta}$.

Next we estimate $I_{4}$,
from \eqref{2.6}$R_{k}=\frac{t^{k-1}}{(k-1)!}\partial^{k-1}_{t}f(0)+\frac{t^{k-1}}{(k-1)!} \ast \partial^{k}_{t}f(t)=R^{1}_{k}+R^{2}_{k},$
so $I_{4}=I^{1}_{4}+I^{2}_{4}$, we have
\begin{equation*}
\begin{split}
I^{1}_{4}
=& \frac{1}{2 \pi i} \int_{\Gamma^{\tau}_{\theta,\kappa}}e^{t_{n}z}K\left(\delta_{\tau}(e^{-(\sigma+z)\tau})\right)\delta_{\tau}(e^{-(\sigma+z)\tau})\tau
   \widetilde{R}^{1}_{k}(e^{-z\tau})dz \\
  & - \frac{1}{2\pi i} \int_{\Gamma_{\theta,\kappa}} {e^{zt_{n}}K(\sigma+z)(\sigma+z)\widehat{R}^{1}_{k} } dz\\
%=& \frac{1}{2 \pi i} \int_{\Gamma^{\tau}_{\theta,\kappa}}e^{t_{n}z}K\left(\delta_{\tau}(e^{-(\sigma+z)\tau})\right)\delta_{\tau}(e^{-(\sigma+z)\tau})\tau
%   \widetilde{R}^{1}_{k}(e^{-z\tau})dz - \frac{1}{2\pi i} \int_{\Gamma^{\tau}_{\theta,\kappa}} {e^{zt_{n}}K(\sigma+z)(\sigma+z)\widehat{R}^{1}_{k} } dz\\
% & - \frac{1}{2\pi i} \int_{\Gamma_{\theta,\kappa}\backslash\Gamma^{\tau}_{\theta,\kappa} } {e^{zt_{n}}K(\sigma+z)(\sigma+z)\widehat{R}^{1}_{k} } dz\\
=& \frac{1}{2 \pi i} \int_{\Gamma^{\tau}_{\theta,\kappa}}e^{t_{n}z}\left(\delta_{\tau}(e^{-(\sigma+z)\tau})+A\right)^{-1}
   \frac{\gamma_{k-1}(e^{-z\tau})}{(k-1)!}\tau^{k}\partial^{k-1}_{t}f(0)dz\\
 & - \frac{1}{2\pi i} \int_{\Gamma^{\tau}_{\theta,\kappa}} {e^{zt_{n}}\left((\sigma+z)^{\gamma}+A\right)^{-1}\frac{\partial^{k-1}_{t}f(0)}{z^{k}}} dz\\
 &  - \frac{1}{2\pi i} \int_{\Gamma_{\theta,\kappa}\backslash\Gamma^{\tau}_{\theta,\kappa}}
     {e^{zt_{n}}\left((\sigma+z)^{\gamma}+A\right)^{-1} \frac{\partial^{k-1}_{t}f(0)}{z^{k}}} dz.
 \end{split}
\end{equation*}
%We use above inequation $\left\|I_{5}\right\| $ and $ \left\|I_{l,3}\right\| $ with $l=k-1$, then
From  Lemma \ref{lemma3.2}, it yields
\begin{equation*}
\left\|I^{1}_{4}\right\| \leq c\tau^{k} t_{n}^{\gamma-1} \left\| \partial^{k-1}_{t}f(0)\right\| .
\end{equation*}
Similarly, using  Lemma \ref{lemma3.3} with $g(t)=\partial^{k}_{t}f(t)$, we get
\begin{equation*}
\left\|I^{2}_{4}\right\| \leq  c \tau^{k}\int^{t_{n}}_{0} (t_{n}-s)^{\gamma-1} \| \partial^{k}_{s}f(s) \| ds.
\end{equation*}
The proof is completed.
\end{proof}

%W(t)= & \frac{1}{2\pi i} \int_{\Gamma_{\theta,\kappa}} {e^{zt}K(\sigma+z)\left( A G(0)+\frac{\sigma+z}{z} f(0)  \right) } dz \\
%      & +\frac{1}{2\pi i}\int_{\Gamma_{\theta,\kappa}}{e^{zt}(\sigma+z)K(\sigma+z)
%      \left(\sum^{k-2}_{l=1}\frac{1}{z^{l+1}}\partial^{l}_{t}f(0)+\widehat{R}_{k}\right)}dz,

%=&\frac{1}{2 \pi i} \int_{\Gamma^{\tau}_{\theta,\kappa}} e^{t_{n}z} K\left(\delta_{\tau}(e^{-(\sigma+z)\tau})\right)\left[\tau^{-1}\mu_{1}
%   (e^{-(\sigma+z)\tau}) AG^{0}+\tau^{-1}\frac{\delta_{\tau}(e^{-(\sigma+z)\tau})}{\delta_{\tau}(e^{-z\tau})}\mu_{2}(e^{-z\tau})f(0)\right]dz \\
%& +\frac{1}{2 \pi i} \int_{\Gamma^{\tau}_{\theta,\kappa}}e^{t_{n}z}K\left(\delta_{\tau}(e^{-(\sigma+z)\tau})\right)\delta_{\tau}(e^{-(\sigma+z)\tau})
%   \sum^{k-2}_{l=1}\left(\frac{\gamma_{l}(e^{-z\tau})}{l!} +\sum^{k-1}_{j=1}d^{(k)}_{l,n}e^{-z j\tau}\right)\tau^{l}\partial^{l}_{t}f(0)dz \\
%& +\frac{1}{2 \pi i} \int_{\Gamma^{\tau}_{\theta,\kappa}}e^{t_{n}z}K\left(\delta_{\tau}(e^{-(\sigma+z)\tau})\right)
%  \delta_{\tau}(e^{-(\sigma+z)\tau})\widetilde{R}_{k}(e^{-z\tau}) dz

\section{Numerical results}
We now numerically verify the above theoretical results including convergence orders  of correction BDF$k$ scheme \eqref{2.4} for \eqref{1.1} in one spatial dimension.
In space direction, it is discretized with the   spectral collocation method  with the Chebyshev-Gauss-Lobatto points \cite{Shen:11} in the interval $\Omega=(-1,1)$.
Since the convergence rate of the spatial discretization is well understood, we focus on the time direction convergence order.
Let us consider the following three examples:
\begin{description}
  \item[(a)] $G_{0}=\sqrt{1-x^2}$ and $ f(x,t)=0$.
%  \item[(b)] $G_{0}=0, f(x,t)=1$.
  \item[(b)] $G_{0}=0$ and $ f(x,t)=(t+1)^5\left(1+\chi_{(0,1)}(x)\right)$.
%  \item[(a)] $G_{0}=x(1-x)}, f(x,t)=0$
  \item[(c)] $G_{0}=\sqrt{1-x^2}$ and $ f(x,t)=\cos(t)\left(1+\chi_{(0,1)}(x)\right)$.
%  \item[(d)] $G_{0}=x(1-x), f(x,t)=xt^{\gamma}$
\end{description}

Since the analytic solutions is unknown, the order of the convergence of the numerical results are computed by the following formula
\begin{equation*}
  {\rm Convergence ~Rate}=\frac{\ln \left(||G^{N/2}-G^{N}||_\infty/||G^{N}-G^{2N}||_\infty\right)}{\ln 2}.
\end{equation*}

\begin{table}[h]\fontsize{8.5pt}{12pt}\selectfont%Éú³É¸¡¶¯±í¸ñ
\begin{center}%\def\tabcolsep{28.5pt}%±í¸ñ¾ÓÖÐ
\caption {The maximum errors and convergent order of correction BDF$k$ scheme \eqref{2.4} for example (a) with $\sigma=0.5$ and $T=1$.} \vspace{5pt}%±êÌ⣬Àë±í¸ñÒ»¶¨µÄ¾àÀë
\begin{tabular}{|c| c | c c c c | c |}                                \hline  %»­¶¥¶ËµÄºáÏß
    $(\alpha,\gamma)$  & \diagbox[dir=SE]{$k$}{$N$} &   40           &   80          &   160          &   320                   & Rate         \\ \hline
              &   2             &   8.7495e-06   &   2.1453e-06   &   5.3116e-07   &   1.3215e-07   &   $\approx$2.0163               \\
              &   3             &   5.0931e-07   &   6.0766e-08   &   7.4234e-09   &   9.1742e-10   &   $\approx$3.0389               \\
(1.7,0.3)    &   4             &   4.2355e-08   &   2.4277e-09   &   1.4544e-10   &   8.9015e-12   &   $\approx$4.0721               \\
              &   5             &   6.9984e-09   &   1.2741e-10   &   3.7184e-12   &   1.1269e-13   &   $\approx$5.3075               \\%%%%%%%%È¡1\720
%              &   5             &   6.9984e-09   &   1.2741e-10   &   3.7184e-12   &   1.1269e-13   &   $\approx$5.3075               \\%%%%%%%%È¡-1\720
              &   6             &   7.8500e-07   &   4.8068e-09   &   1.5732e-13   &   2.7200e-15   &   $\geq$6.0000                 \\
              \hline % »­µ×¶ËµÄºáÏß
              &   2             &   2.7568e-05   &   6.8210e-06   &   1.6966e-06   &   4.2309e-07   &   $\approx$2.0086               \\
              &   3             &   1.0888e-06   &   1.3273e-07   &   1.6383e-08   &   2.0350e-09   &   $\approx$3.0212               \\
  (1.3, 0.7) &   4             &   6.4696e-08   &   3.8130e-09   &   2.3168e-10   &   1.4275e-11   &   $\approx$4.0486               \\
              &   5             &   6.0435e-09   &   1.6306e-10   &   4.8147e-12   &   1.2657e-13   &   $\approx$5.1811               \\%%%%%%%%È¡1\720
%             &   5             &   6.0435e-09   &   1.6306e-10   &   4.8147e-12   &   1.2657e-13   &   $\approx$5.1811               \\%%%%%%%%È¡-1\720
              &   6             &   2.9492e-05   &   3.3112e-08   &   2.3309e-13   &   5.1903e-14   &   $\geq$6.0000                 \\
              \hline % »­µ×¶ËµÄºáÏß
%%%%%%%%%%%%%%%%%%%%%%%%%%%%%%%%%%%%%%%%%%%%%%%%%%%%%%%%%%%%%%%%%%%%%%%%%%%%%%%%%%%%%%%%%%%%%%%%%%%%%%%%%%%%%%%%%%%%%%%%%%%%%%%%%%%
\end{tabular}\label{table:a}%\vspace{-15pt}
\end{center}
\end{table}

\begin{table}[h]\fontsize{8.5pt}{12pt}\selectfont%Éú³É¸¡¶¯±í¸ñ
\begin{center}%\def\tabcolsep{28.5pt}%±í¸ñ¾ÓÖÐ
\caption {The maximum errors and convergent order of correction BDF$k$ scheme \eqref{2.4} for example (b)  with $\sigma=0.5$ and $T=1$.} \vspace{5pt}%±êÌ⣬Àë±í¸ñÒ»¶¨µÄ¾àÀë
\begin{tabular}{|c| c | c c c c | c |}                                \hline  %»­¶¥¶ËµÄºáÏß
    $(\alpha,\gamma)$  & \diagbox[dir=SE]{$k$}{$N$} &   40           &   80          &   160          &   320                   & Rate         \\ \hline
              &   2             &   1.2062e-03   &   3.0857e-04   &   7.8041e-05   &   1.9624e-05   &   $\approx$1.9806               \\
              &   3             &   5.5613e-05   &   7.1582e-06   &   9.0795e-07   &   1.1432e-07   &   $\approx$2.9754              \\
   (1.7,0.3)  &   4             &   2.1564e-06   &   1.3904e-07   &   8.8267e-09   &   5.5601e-10   &   $\approx$3.9738             \\
%              &   5             &   6.5543e-08   &   2.2094e-09   &   7.2298e-11   &   2.4247e-12   &   $\approx$4.9075               \\%%%%%%%%È¡1\720
              &   5             &   6.4535e-08   &   2.1478e-09   &   6.8484e-11   &   2.1938e-12   &   $\approx$4.9481               \\%%%%%%%%È¡-1\720
              &   6             &   7.8075e-07   &   2.3185e-09   &   5.7554e-13   &   4.2633e-14   &   $\geq$6.0000                 \\
              \hline % »­µ×¶ËµÄºáÏß
              &   2             &   9.0605e-03   &   2.3200e-03   &   5.8706e-04   &   1.4766e-04   &   $\approx$1.9798             \\
              &   3             &   4.3627e-04   &   5.6219e-05   &   7.1349e-06   &   8.9864e-07   &   $\approx$2.9744               \\
   (1.3, 0.7) &   4             &   1.7567e-05   &   1.1341e-06   &   7.2043e-08   &   4.5394e-09   &   $\approx$3.9727               \\
%              &   5             &   5.8365e-07   &   1.8912e-08   &   6.2315e-10   &   2.1533e-11   &   $\approx$4.9088               \\%%%%%%%%È¡1\720
              &   5             &   6.0435e-09   &   1.6306e-10   &   4.8147e-12   &   1.2657e-13   &   $\approx$5.1811               \\%%%%%%%%È¡-1\720
              &   6             &   4.6825e-05   &   5.7562e-08   &   5.6595e-12   &   1.1688e-12   &   $\geq$6.0000                 \\
              \hline % »­µ×¶ËµÄºáÏß
%%%%%%%%%%%%%%%%%%%%%%%%%%%%%%%%%%%%%%%%%%%%%%%%%%%%%%%%%%%%%%%%%%%%%%%%%%%%%%%%%%%%%%%%%%%%%%%%%%%%%%%%%%%%%%%%%%%%%%%%%%%%%%%%%%%
\end{tabular}\label{table:b}%\vspace{-15pt}
\end{center}
\end{table}

\begin{table}[h]\fontsize{8.5pt}{12pt}\selectfont%Éú³É¸¡¶¯±í¸ñ
\begin{center}%\def\tabcolsep{28.5pt}%±í¸ñ¾ÓÖÐ
\caption {The maximum errors and convergent order of correction BDF$k$ scheme \eqref{2.4} for example (c)  with $\sigma=0.5$ and $T=1$.} \vspace{5pt}%±êÌ⣬Àë±í¸ñÒ»¶¨µÄ¾àÀë
\begin{tabular}{|c| c | c c c c | c |}                                \hline  %»­¶¥¶ËµÄºáÏß
    $(\alpha,\gamma)$  & \diagbox[dir=SE]{$k$}{$N$} &   40           &   80          &   160          &   320                   & Rate         \\ \hline
              &   2             &   8.5901e-06   &   2.1399e-06   &   5.3403e-07   &   1.3339e-07   &   $\approx$2.0030               \\
              &   3             &   1.0845e-07   &   1.3026e-08   &   1.5976e-09   &   1.9785e-10   &   $\approx$3.0328              \\
   (1.7,0.3)  &   4             &   8.3418e-09   &   4.6450e-10   &   2.7389e-11   &   1.6622e-12   &   $\approx$4.0977             \\
              &   5             &   2.2068e-09   &   3.2239e-11   &   9.4552e-13   &   2.9421e-14   &   $\approx$5.3982               \\%%%%%%%%È¡1\720
%              &   5             &   2.2068e-09   &   3.2239e-11   &   9.4552e-13   &   2.9421e-14   &   $\approx$5.3982               \\%%%%%%%%È¡-1\720
              &   6             &   2.9255e-07   &   2.5148e-09   &   4.6241e-14   &   1.4988e-15   &   $\geq$6.0000                 \\
              \hline % »­µ×¶ËµÄºáÏß
              &   2             &   2.8858e-05   &   7.4063e-06   &   1.8749e-06   &   4.7159e-07   &   $\approx$1.9784             \\
              &   3             &   2.5394e-06   &   3.0437e-07   &   3.7255e-08   &   4.6082e-09   &   $\approx$3.0354               \\
   (1.3, 0.7) &   4             &   1.8942e-07   &   1.1115e-08   &   6.7379e-10   &   4.1483e-11   &   $\approx$4.0522               \\
              &   5             &   1.4510e-08   &   4.3197e-10   &   1.2675e-11   &   3.6726e-13   &   $\approx$5.0899               \\%%%%%%%%È¡1\720
%              &   5             &   1.4510e-08   &   4.3197e-10   &   1.2675e-11   &   3.6726e-13   &   $\approx$5.0899               \\%%%%%%%%È¡-1\720
              &   6             &   1.6651e-05   &   2.4673e-08   &   4.4675e-13   &   2.4092e-14   &   $\geq$6.0000                 \\
              \hline % »­µ×¶ËµÄºáÏß
%%%%%%%%%%%%%%%%%%%%%%%%%%%%%%%%%%%%%%%%%%%%%%%%%%%%%%%%%%%%%%%%%%%%%%%%%%%%%%%%%%%%%%%%%%%%%%%%%%%%%%%%%%%%%%%%%%%%%%%%%%%%%%%%%%%
\end{tabular}\label{table:c}%\vspace{-15pt}
\end{center}
\end{table}

\begin{table}[h]\fontsize{8.5pt}{12pt}\selectfont%Éú³É¸¡¶¯±í¸ñ
\begin{center}%\def\tabcolsep{28.5pt}%±í¸ñ¾ÓÖÐ
\caption {The maximum errors and convergent order of stand BDF$k$ scheme \eqref{2.3} for example (a)  with $\sigma=0.5$ and $T=1$.} \vspace{5pt}%±êÌ⣬Àë±í¸ñÒ»¶¨µÄ¾àÀë
\begin{tabular}{|c| c | c c  c c | c |}                                \hline  %»­¶¥¶ËµÄºáÏß
    $(\alpha,\gamma)$  & \diagbox[dir=SE]{$k$}{$N$} &   40           &   80          &   160          &   320                    & Rate         \\ \hline
              &   2             &   1.9994e-04   &   9.9876e-05   &   4.9910e-05   &   2.4948e-05   &   $\approx$1.0009                \\
              &   3             &   2.0105e-04   &   1.0014e-04   &   4.9974e-05   &   2.4963e-05   &   $\approx$1.0032                \\
  (1.7,0.3)  &   4             &   2.0100e-04   &   1.0013e-04   &   4.9973e-05   &   2.4963e-05   &   $\approx$1.0031                \\
              &   5             &   2.0101e-04   &   1.0013e-04   &   4.9973e-05   &   2.4963e-05   &   $\approx$1.0031                \\
              &   6             &   2.0585e-04   &   1.0014e-04   &   4.9973e-05   &   2.4963e-05   &   $\approx$1.0146                 \\
              \hline % »­µ×¶ËµÄºáÏß
              &   2             &   8.1394e-04   &   4.0421e-04   &   2.0145e-04   &   1.0057e-04   &     $\approx$1.0055            \\
              &   3             &   8.0891e-04   &   4.0313e-04   &   2.0121e-04   &   1.0051e-04   &     $\approx$1.0029             \\
 (1.3, 0.7) &   4             &   8.0950e-04   &   4.0320e-04   &   2.0121e-04   &   1.0051e-04   &     $\approx$1.0033             \\
              &   5             &   8.0946e-04   &   4.0319e-04   &   2.0121e-04   &   1.0051e-04   &     $\approx$1.0032             \\
              &   6             &   8.1383e-04   &   4.0318e-04   &   2.0121e-04   &   1.0051e-04   &     $\approx$1.0058               \\
              \hline % »­µ×¶ËµÄºáÏß
%%%%%%%%%%%%%%%%%%%%%%%%%%%%%%%%%%%%%%%%%%%%%%%%%%%%%%%%%%%%%%%%%%%%%%%%%%%%%%%%%%%%%%%%%%%%%%%%%%%%%%%%%%%%%%%%%%%%%%%%%%%%%%%%%%%
\end{tabular}\label{table:aa}%\vspace{-15pt}
\end{center}
\end{table}

\begin{table}[h]\fontsize{8.5pt}{12pt}\selectfont%Éú³É¸¡¶¯±í¸ñ
\begin{center}%\def\tabcolsep{28.5pt}%±í¸ñ¾ÓÖÐ
\caption {The maximum errors and convergent order of stand BDF$k$ scheme \eqref{2.3} for example (c)  with $\sigma=0.5$ and $T=1$.} \vspace{5pt}%±êÌ⣬Àë±í¸ñÒ»¶¨µÄ¾àÀë
\begin{tabular}{|c| c | c c c c | c |}                                \hline  %»­¶¥¶ËµÄºáÏß
    $(\alpha,\gamma)$  & \diagbox[dir=SE]{$k$}{$N$} &   40           &   80          &   160          &   320                   & Rate         \\ \hline
              &   2             &   9.4869e-05   &   4.8909e-05   &   2.4823e-05   &   1.2503e-05   &   $\approx$0.9745               \\
              &   3             &   1.0120e-04   &   5.0491e-05   &   2.5218e-05   &   1.2602e-05   &   $\approx$1.0018              \\
   (1.7,0.3)  &   4             &   1.0120e-04   &   5.0492e-05   &   2.5218e-05   &   1.2602e-05   &   $\approx$1.0019             \\
              &   5             &   1.0120e-04   &   5.0492e-05   &   2.5218e-05   &   1.2602e-05   &   $\approx$1.0018               \\
              &   6             &   1.0363e-04   &   5.0496e-05   &   2.5218e-05   &   1.2602e-05   &   $\approx$1.0133                 \\
              \hline % »­µ×¶ËµÄºáÏß
              &   2             &   3.4862e-04   &   1.5918e-04   &   7.5788e-05   &   3.6941e-05   &   $\approx$1.0794             \\
              &   3             &   2.9394e-04   &   1.4540e-04   &   7.2331e-05   &   3.6076e-05   &   $\approx$1.0088               \\
   (1.3, 0.7) &   4             &   2.9348e-04   &   1.4535e-04   &   7.2325e-05   &   3.6075e-05   &   $\approx$1.0081               \\
              &   5             &   2.9350e-04   &   1.4535e-04   &   7.2325e-05   &   3.6075e-05   &   $\approx$1.0081               \\
              &   6             &   3.0291e-04   &   1.4536e-04   &   7.2325e-05   &   3.6075e-05   &   $\approx$1.0233                 \\
              \hline % »­µ×¶ËµÄºáÏß
%%%%%%%%%%%%%%%%%%%%%%%%%%%%%%%%%%%%%%%%%%%%%%%%%%%%%%%%%%%%%%%%%%%%%%%%%%%%%%%%%%%%%%%%%%%%%%%%%%%%%%%%%%%%%%%%%%%%%%%%%%%%%%%%%%%
\end{tabular}\label{table:cc}%\vspace{-15pt}
\end{center}
\end{table}
Tables \ref{table:aa} and \ref{table:cc} show that the stand  BDF$k$ scheme  in \eqref{2.3} just  achieves the  first-order convergence  and
the corrected BDF$k$ scheme in \eqref{2.4} preserves the high-order convergence rate with the nonsmooth data.

\section*{Acknowledgments}
The  authors are grateful to Professor Martin Stynes and Dr. Zhi Zhou for them  valuable comments.

\bibliographystyle{amsplain}

\end{document}